\documentclass[nolayout]{article}

\usepackage[utf8]{inputenc}
\usepackage[english]{babel}
\usepackage{amsmath,amsthm}
\usepackage{geometry}
\usepackage{tikz}
\usetikzlibrary{fit,positioning}
\usepackage[textwidth=4cm,textsize=footnotesize]{todonotes}
\usepackage{fancyhdr}
\usepackage{amssymb,color,bbm,xargs}
\usepackage{graphicx}
\usepackage[active]{srcltx}
\usepackage{ifthen}
\usepackage{enumerate}
\usepackage[shortlabels]{enumitem}
\usepackage{dsfont}
\usepackage{subcaption}
\usepackage{authblk}
\usepackage{hyperref}
\hypersetup{
    bookmarks=true,       
    colorlinks=true,     
    linkcolor=red,        
    citecolor=green,        
    filecolor=magenta,    
    urlcolor=cyan           
}

\newtheorem{lemma}{Lemma}
\newtheorem{proposition}[lemma]{Proposition}
\newtheorem{theorem}[lemma]{Theorem}
\newtheorem{corollary}[lemma]{Corollary}

\definecolor{lavander}{cmyk}{0,0.48,0,0}
\definecolor{violet}{cmyk}{0.79,0.88,0,0}
\definecolor{burntorange}{cmyk}{0,0.52,1,0}
\definecolor{burntgreen}{cmyk}{0.62,0.44,0.47,0}
\definecolor{burntblue}{cmyk}{0.86,0.30,0.18,0}
\definecolor{palegreen}{cmyk}{0.86,0.30,0.96,0}

\def\sup{\mathrm{sup}}

\def\1{\mathds{1}}
\def\rset{\mathbb{R}}
\def\rmd{\mathrm{d}}
\def\eqsp{\,}
\def\Pstar{\mathbb{P}_{\pi_{\star}}}
\def\bayes{\pi_{\star}}

\def\pseudo_dist{d}
\newcommand{\limit}[1]{\underset{#1\to \infty}{\longrightarrow}}
\newcommand{\E}{\mathbb{E}}
\newcommand{\kullback}{\mathsf{L}}

\newcommand{\un}[1]{\mathds{1}_{#1}}

\newcommand{\pa}[1]{\left(#1\right)}
\newcommand{\cro}[1]{\left[#1\right]}
\newcommand{\absj}[1]{\left|#1\right|}
\newcommand{\norm}[1]{\left\|#1\right\|}

\newcommand{\bP}{\mathbb{P}}

\newcommand{\bZ}{\mathbb{Z}}

\newcommand{\cA}{\mathcal{A}}

\newcommand{\cV}{\mathcal{V}}
\newcommand{\cX}{\mathcal{X}}

\newcommand{\cF}{\mathcal{F}}
\newcommand{\rme}{\mathrm{e}}

\newcommand{\sfE}{\mathsf{E}}

\newcommand{\sfN}{\mathsf{N}}
\newcommand{\sfP}{\mathsf{P}}
\newcommand{\sfR}{\mathsf{R}}
\newcommand{\sfX}{\mathsf{X}}
\newcommand{\sfS}{\mathsf{S}}

\newcommand{\condlik}{\mathsf{K}}

\newcommand{\shift}{\vartheta}

\newcounter{hypH}
\newenvironment{hypH}{\refstepcounter{hypH}\begin{itemize}
\item[{\bf H\arabic{hypH}}]}{\end{itemize}}

\newcounter{hypA}

\begin{document}

\title{
A Bayesian nonparametric approach for generalized Bradley-Terry models in random environment}
\date{} 

\author[1]{Sylvain Le Corff}
\author[1]{Matthieu Lerasle}
\author[2]{Elodie Vernet}
\affil[1]{{\small Laboratoire de Math\'ematiques d'Orsay, Univ. Paris-Sud, CNRS, Universit\'e Paris-Saclay.}}
\affil[2]{{\small Centre de Math\'ematiques Appliqu\'ees, \'Ecole Polytechnique, Palaiseau.}}

\lhead{Le Corff, S., Lerasle, M. and Vernet, E.}
\rhead{}

\maketitle

\begin{abstract}
This paper deals with the estimation of the unknown distribution of hidden random variables from the observation of pairwise comparisons between these variables. 
This problem is inspired by recent developments on Bradley-Terry models in random environment \cite{CheDielLer:2017} since this framework happens to be relevant to predict for instance the issue of a championship from the observation of a few contests per team \cite{diel:lecorff:lerasle:2018}.
This paper provides three contributions on a Bayesian nonparametric approach to solve this problem.
First, we establish contraction rates of the posterior distribution. We also propose a Markov Chain Monte Carlo  algorithm to approximately sample from this posterior distribution inspired from a recent Bayesian nonparametric method for hidden Markov models. Finally, the performance of this algorithm are appreciated by comparing predictions on the issue of a championship based on the actual values of the teams and those obtained by sampling from the estimated posterior distribution.
\end{abstract}

\section{Introduction}
This paper considers a class of models to deal with pairwise comparison of individuals. The outcome of all the successive comparisons is described by an ordering after all the contestants have met once.
A potentially large number of participants are ranked using scores which aggregate the outcomes of their meetings.
The objective is to predict this ranking when only a small number of pairwise comparisons has been observed.
In generalized Bradley-Terry models, each player $i\geqslant 1$ is associated with an unknown real parameter $V_i$ characterizing his strength or ability. 
When contestants $i$ and $j$ face each other, the outcome of their meeting is a random variable taking values in a discrete set $\cX$. For all $x\in\cX$, this random variable equals $x$ with probability $\condlik(x,V_i,V_j)$, where the function $\condlik$ is known.
The most famous example of such a function $\condlik$ is the Bradley-Terry model \cite{bradley:terry:1952,zemerlo:1929} in which 
\begin{equation}\label{eq:defBT}
 \sfX=\{0,1\},\qquad \condlik(x,V_i,V_j)=\frac{V_i^xV_j^{1-x}}{V_i+V_j}\eqsp.
\end{equation}
If $x=1$, $i$ has beaten $j$, otherwise, $j$ has beaten $i$. These models have emerged as key tools to describe pairwise comparisons and have been applied to chess ranking \cite{joe:1990}, sports \cite{Sir_Red:2009}, animal behaviour \cite{whiting:stuartfox:oconnor:firth:bennett:bloomberg:2006}.
The maximum likelihood estimator (MLE) of the strengths $V_i$ was studied theoretically in \cite{Simons_Yao:1999} when each contestant has faced all the others once.
This result has led to interesting developments in computational statistics to design efficient algorithms to approximate the MLE \cite{caron:doucet:2012,hunter:2004}. The assumption that each contestant has faced all the others is restrictive in some applications and there has been several attempt to weaken it \cite{diel:lecorff:lerasle:2018,YanYangXu:2011}. 
However, it cannot be completely relaxed as the MLE only exists in the Bradley-Terry model if there exists a path between each pair of players in the oriented graph where an edge is drawn from $i$ to $j$ if $i$ has beaten $j$ \cite{zemerlo:1929}. When only a few contests have been observed, it is very likely that the previous condition is not met.

Studying the scores in the American baseball league, \cite{Sir_Red:2009} introduced a simple model where the parameters $V_i$ are independent and identically distributed (i.i.d.) with common uniform distribution on $[\alpha,1]$ for some parameter $\alpha>0$ that has to be estimated. The authors obtained  striking results for the prediction of consecutive-game team winning and losing streaks for instance. This motivated in \cite{CheDielLer:2017} the introduction of the Bradley-Terry model in random environment where the function $\condlik$ is given by \eqref{eq:defBT} and the strengths $(V_i)_{i\geqslant 1}$ are assumed i.i.d. with common distribution $\pi$ on $(0,+\infty)$.
The authors proved in \cite{CheDielLer:2017} that in the case of general Bradley-Terry models in random environment, the winner of a championship is the one with maximal strength if and only if the tail of $\pi$ is sufficiently convex.
In other words, interesting predictions regarding the outcome of a tournament can be inferred from $\pi$.
This is why \cite{diel:lecorff:lerasle:2018} considered the problem of estimating the distribution of the $(V_i)_{i\geqslant 1}$ in generalized Bradley-Terry models in random environment, showing that the MLE of this distribution can be defined even when each player has only been involved in $2$ contests.
They also studied the MLE of $\pi$ when players meet according to the round-robin scheduling, a widely spread method employed by tournament schedulers.
The cornerstone of the work presented in \cite{diel:lecorff:lerasle:2018} is that risk bounds for the MLE of $\pi$ follow from the analysis of the likelihood of the following graphical model \cite[Lemma 1]{diel:lecorff:lerasle:2018}.
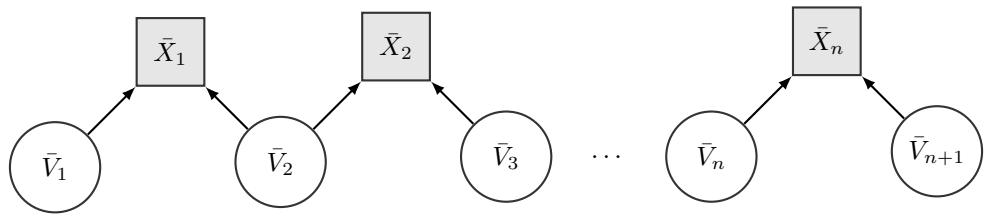
\begin{figure}[h!]
\centering
\begin{tikzpicture}
\tikzstyle{player}=[circle, minimum size = 12mm, thick, draw =black!80, node distance = 8mm]
\tikzstyle{game}=[rectangle, minimum size = 9mm, thick, draw =black!80, node distance = 9mm]
\tikzstyle{connect}=[-latex, thick]
\tikzstyle{box}=[rectangle, draw=black!100]
\node[player] (v0){$\bar V_{1}$};
\node[game, fill = black!10] (X01) [above right =of v0] {$\bar X_{1}$};
\node[player] (v1) [below right=of X01] {$\bar V_{2}$ };

\path (v0) edge [connect] (X01);  
\path (v1) edge [connect] (X01);  

\node[game, fill = black!10] (X12) [above right =of v1] {$\bar X_{2}$};
\node[player] (v2) [below right=of X12] {$\bar V_{3}$ };

\path (v1) edge [connect] (X12);  
\path (v2) edge [connect] (X12);  

\node[player] (vq) [right=of v2,xshift = .7cm] {$\bar V_{n}$ };

\path (v2) -- node[auto=false]{\ldots}  (vq);

\node[game, fill = black!10] (Xq1) [above right =of vq] {$\bar X_{n}$};
\node[player] (vq1) [below right=of Xq1] {$\bar V_{n+1}$ };

\path (vq1) edge [connect] (Xq1); 
\path (vq) edge [connect] (Xq1); 

\end{tikzpicture}
\caption{Graphical model of the nonparametric hidden graph.}
\label{fig:generic:graphicalmodel:intro}
\end{figure}

The details of the link between the graphical model in Figure~\ref{fig:generic:graphicalmodel:intro} and (generalized) Bradley-Terry models can be found in \cite[Section 2]{diel:lecorff:lerasle:2018}. Roughly speaking, in Figure~\ref{fig:generic:graphicalmodel:intro}, each $\bar V_i$, $1\leqslant i \leqslant n+1$,  gathers a group $(V_j)_{j\in G_i}$ of contestants in the Bradley-Terry model and each $\bar X_i$, $1\leqslant i \leqslant n$, gathers the outcomes of meetings between contestants $j$ and $k$ when both belong to $G_i$ or when $j\in G_i$ and $k\in G_{i+1}$.
In this representation, it is shown in \cite{diel:lecorff:lerasle:2018} that when meetings are scheduled according to the round-robin algorithm the Markov chain $(\bar X_{i},\bar V_{i+1})_{1\leqslant i\leqslant n}$ is stationary.

The contribution of this paper can be summarized as follows.
First, in a Bayesian formulation of the problem considered in \cite{diel:lecorff:lerasle:2018}, Section~\ref{sec:main} provides the first nonparametric posterior concentration rates for the estimation of the distribution of the hidden variables $(V_i)_{1\leqslant i\leqslant n+1}$ associated with the graphical model given in Figure~\ref{fig:generic:graphicalmodel:intro}. Then, Section~\ref{sec:exp} introduces an algorithm to simulate this posterior distribution. In particular, this algorithm allows to compute an estimator of $\pi$ and to observe the actual performance of our estimation strategy.  These practical aspects are illustrated in a simulation study inspired by data from season 2017-2018 of the French Ligue 1 in Section~\ref{sec:exp}.
 
 In the case of i.i.d. observations, posterior concentration rates have been obtained for a large class of prior distributions and target densities, see
 \cite{kruijer:rousseau:vandervaart:2010,shen:tokdar:ghosal:2013} for Gaussian mixtures and \cite{ghosal:2001,rousseau:2010} for Beta mixtures. Overviews on  recent developments of the nonparametric Bayesian theory can be found in \cite{rousseau:2016, ghosal:vandervaart:2017}  and the references therein. 
 Posterior concentration rates were established for partially observed dependent data by \cite{vernet:2015}  following the seminal paper \cite{ghosal:vandervaart:2007}. 
 In \cite{vernet:2015}, the observations arise from a discrete state space hidden Markov model and minimax posterior concentration rates are obtained for the finite-dimensional distributions of the observations. 
 In this paper, posterior concentration rates are obtained for the unknown distribution of latent variables which take values in a continuous state space. 
 Following the approach of \cite{ghosal:vandervaart:2007} (see also \cite{rousseau:2016}) this requires in particular to design exponentially consistent tests by using concentration inequalities for Bradley-Terry models. 
An important difference with the frequentist analysis performed in \cite{diel:lecorff:lerasle:2018} is that these inequalities have to be uniform in $\pi$ to bound the power of these tests. 
The main results given in Section \ref{sec:main} require therefore a substantial extension of the technical tools used in \cite{diel:lecorff:lerasle:2018} even if the underlying strategy is similar. 
 
The Bayesian nonparametric procedure introduced in this paper to sample from the posterior distribution of $\pi$ is inspired by the mixture of Dirichlet processes and the data augmentation scheme presented in \cite{yau:papaspiliopoulos:roberts:holmes:2011}. The algorithm proposed by \cite{yau:papaspiliopoulos:roberts:holmes:2011} is a Markov Chain Monte Carlo (MCMC) algorithm applied for discrete hidden Markov models with conditional distribution specified as a mixture model in which a base measure is mixed with respect to a Dirichlet process prior (DPP). In Section~\ref{sec:exp}, we propose a mixture of Dirichlet processes formulation where the unknown distribution $\pi$ is defined as a mixture model where Gaussian distributions are mixed with respect to a DPP. The joint posterior distribution of the parameters of both the DPP and the hidden variables $(V_i)_{1\leqslant i\leqslant n+1}$ is sampled from by adapting the block Gibbs sampling mechanism of \cite{yau:papaspiliopoulos:roberts:holmes:2011} to our setting. The main challenge is to sample from the posterior distribution of the states $(V_i)_{1\leqslant i\leqslant n+1}$. Contrary to the discrete case setting considered in \cite{yau:papaspiliopoulos:roberts:holmes:2011}, exact sampling from this distribution is not  possible here and a Sequential Monte Carlo algorithm is used instead. 
Section~\ref{sec:exp} presents the performance of the overall procedure when the Forward Filtering Backward Simulation algorithm proposed by \cite{godsill:doucet:west:2004} (see also \cite{douc:garivier:moulines:olsson:2011}) to perform this step is used.

The remaining of the paper is divided as follows. Section~\ref{sec:setting} displays the formal setting of the paper and Section~\ref{sec:main} the main results: posterior concentration rates are established under standard assumptions on the prior distribution. Examples of priors satisfying these assumptions are also given. Section~\ref{sec:exp} presents the numerical experiments. The MCMC algorithm to sample from the posterior distribution is introduced and an application to a simulated dataset inspired from the results of the French soccer championship is also discussed.

\section{Setting}
\label{sec:setting}
Let $n$ be a positive integer and $(V_i)_{1\leqslant i \leqslant n+1}$ denote i.i.d. random variables taking values in a measurable set $(\cV,B(\cV))$. For all $1\leqslant i \leqslant n$, the observation $X_{i}$ takes values in a discrete set $\cX$ and conditionally on $(V_i)_{1\leqslant i \leqslant n+1}$ the random variables $(X_{i})_{1\leqslant i\leqslant n}$ are independent. The conditional distribution of $X_{i}$ given $(V_k)_{1\leqslant k \leqslant n+1}$ depends on $V_i$ and $V_{i+1}$ only and  is denoted by $x\mapsto \condlik(x,V_i,V_{i+1})$, where $\condlik: \cX\times\cV\times\cV\to[0,1]$ is known. 
Let $\Pi$ be a set of probability measures on $(\cV,B(\cV))$. For all $\pi\in\Pi$, note that when the probability distribution of each $V_i$, $1\leqslant i \leqslant n+1$, is $\pi$, then the joint sequence $(X_{i},V_{i+1})_{1\leqslant i\leqslant n}$ is a stationary Markov chain with transition kernel $P:\cX\times \cV \to \mathcal{P}(\cX)\otimes B(\cV)$ given, for all $(x,v)\in \cX\times\cV$ and all $(x',A)\in \cX\times B(\cV)$, by
\begin{equation}
\label{eq:kernel:model}
P_\pi(x,v;x',A) = \int \1_{A}(v')\pi(\rmd v') \condlik(x',v,v')\eqsp,
\end{equation}
where $\un{A}$ denotes the indicator function of the set $A$. This Markov chain may be extended to a stationary process $(X_i,V_{i+1})_{i\in \bZ}$ with the same transition kernel $P_\pi$. In the following, we use the shorthand notation $X = (X_i)_{i\in \bZ}$ and $V = (V_i)_{i\in \bZ}$. The joint law of this extended joint process is denoted by $\bP_{\pi}$ and the associated expectation is written $\E_{\pi}$. For all $1\leqslant i \leqslant n$, and all $\pi\in\Pi$,
\[
\bP_{\pi}\left(X_{i}\middle| V,X_{1:i-1}\right) = \bP_{\pi}\left(X_{i}\middle| V_{i},V_{i+1}\right)=\condlik\left(X_{i},V_i,V_{i+1}\right)\eqsp,
\]
where for any sequence $(a_{\ell})_{\ell\in\mathbb{Z}}$, $a_{u:v} = (a_u,\ldots,a_v)$ if $u\le v$ and $a_{u:v} = \emptyset$ if $u>v$.
These conditional dependences are represented in the graphical model given in Figure~\ref{fig:generic:graphicalmodel}. 
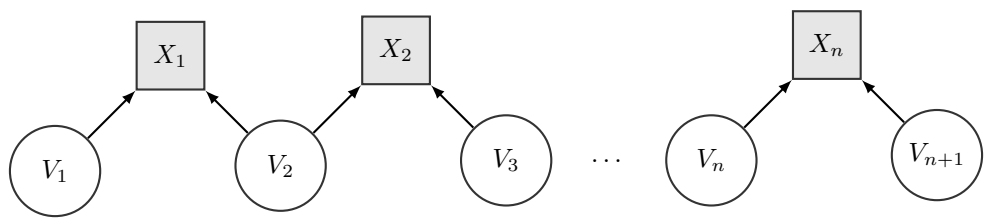
\begin{figure}[h!]
\centering
\begin{tikzpicture}
\tikzstyle{player}=[circle, minimum size = 12mm, thick, draw =black!80, node distance = 8mm]
\tikzstyle{game}=[rectangle, minimum size = 9mm, thick, draw =black!80, node distance = 9mm]
\tikzstyle{connect}=[-latex, thick]
\tikzstyle{box}=[rectangle, draw=black!100]
\node[player] (v0){$V_{1}$};
\node[game, fill = black!10] (X01) [above right =of v0] {$X_{1}$};
\node[player] (v1) [below right=of X01] {$V_{2}$ };

\path (v0) edge [connect] (X01);  
\path (v1) edge [connect] (X01);  

\node[game, fill = black!10] (X12) [above right =of v1] {$X_{2}$};
\node[player] (v2) [below right=of X12] {$V_{3}$ };

\path (v1) edge [connect] (X12);  
\path (v2) edge [connect] (X12);  

\node[player] (vq) [right=of v2,xshift = .7cm] {$V_{n}$ };

\path (v2) -- node[auto=false]{\ldots}  (vq);

\node[game, fill = black!10] (Xq1) [above right =of vq] {$X_{n}$};
\node[player] (vq1) [below right=of Xq1] {$V_{n+1}$ };

\path (vq1) edge [connect] (Xq1); 
\path (vq) edge [connect] (Xq1); 

\end{tikzpicture}
\caption{Graphical model of the nonparametric hidden graph.}
\label{fig:generic:graphicalmodel}
\end{figure}
In the Bayesian setting of this paper, the set of distributions $\Pi$ is equipped with a sigma-algebra $\cA$ and a prior distribution $\mu$. The posterior distribution of any $\pi\in\Pi$ given the observations is defined, for all $A\in\cA$, by
\[
\mu\left(  \pi \in A  \middle| X_{1:n} \right)
=
\frac{\int\prod_{i=1}^{n} \condlik(X_i,v_i,v_{i+1}) \un{A}(\pi) \pi^{\otimes n+1}(\rmd v_{1:n+1}) \mu(\rmd \pi)}{\int \prod_{i=1}^{n} \condlik(X_i,v_i,v_{i+1}) \pi^{\otimes n+1}(\rmd v_{1:n+1}) \mu(\rmd \pi)}\eqsp.
\]
In the following, $\bayes$ denotes the true value of the unknown density. This paper focuses on the frequentist properties of the posterior distribution $\mu\left(  \cdot  \middle| X_{1:n} \right)$. The aim is to establish rates at which the posterior distribution concentrates around $\bayes$ for a well suited loss function. In the following, let $\gamma$ be a $\sigma$-finite distribution on $(\cV,B(\cV))$ and $\sfS^+$ be the set of probability densities with respect to $\gamma$. $\sfS^+$ is equipped with the topology induced by the $\mathrm{L}_1$-norm and the corresponding Borel $\sigma$-field is written $B(\sfS^+)$.

\section{Posterior concentration rates}\label{sec:main}
\subsection{Main results}
Consider the following assumption.
\begin{hypH}
\label{assum:lowerbound}
There exists $0<\nu<1$ such that for all $x\in\cX$ and all $v,w\in\cV$, $\condlik(x,v,w)\ge \nu$.
\end{hypH}
Under H\ref{assum:lowerbound}, the transition kernel of the Markov chain $(V_{i+1},X_{i})_{i\in \mathbb{Z}}$ satisfies, for all $i\in \mathbb{Z}$ and all $A\in B(\cV)$,
\begin{equation}
\label{eq:strong:mixing}
P_{\pi}(X_{i-1},V_i;X_i,A) =\int\un{A}(v_{i+1})\pi(\rmd v_{i+1})\condlik(X_i,V_i,v_{i+1})\geqslant \nu \pi(A)\eqsp.
\end{equation}
This uniform lower bound ensures that the joint Markov chain $(V_{i+1},X_{i})_{i\in \mathbb{Z}}$ is uniformly ergodic and that the whole space $\cV\times\cX$ is a small set. 
Theorem~\ref{th:thetheorem} establishes the posterior concentration rates under $\bP_{\bayes}$, it is derived from two important results.

 Following \cite{rousseau:2016} and the references therein, Kullback-Leibler type conditions are first required. Kullback-Leibler controls are obtained in Proposition \ref{prop:controle_KL_et_var_ln}. For any $\pi\in\Pi$ and any $x_{1:n}\in\cX^n$, let $\ell_n(\pi,x_{1:n})$ be the loglikelihood of the observations $x_{1:n}$,
\begin{equation}
\label{eq:loglik}
\ell_n(\pi,x_{1:n}) = \log \left(\int\prod_{i=1}^{n} \condlik(x_i,v_i,v_{i+1}) \pi^{\otimes n+1}(\rmd v_{1:n+1})\right)\eqsp.
\end{equation}
Let $\|\cdot\|_{\mathsf{tv}}$ be the total variation distance between probability measures: for all $\pi_1,\pi_2\in\Pi$,
\[
\|\pi_1-\pi_2\|_{\mathsf{tv}} = \mathrm{sup}_{A\in B(\cV) }\left|\pi_1(A)-\pi_2(A)\right|\eqsp.
\]
\begin{proposition}
\label{prop:controle_KL_et_var_ln}
Assume that  H\ref{assum:lowerbound} holds. Then, there exists a constant $c_{\nu}>0$ such that for all distributions $\pi,\pi'\in\Pi$ and all $ n\geqslant 1$,
\begin{equation}\label{eq:control_KL}
\E_{\pi}\left[  { \ell_n(\pi,X_{1:n}) -  \ell_n(\pi',X_{1:n})} \right] \\
\leqslant c_{\nu}\|\pi-\pi'\|_{\mathsf{tv}}^2\eqsp n 
\end{equation}
and
\begin{equation}\label{eq:control_KL2}
\mathrm{Var}_{\pi}\left[ {\ell_n(\pi,X_{1:n}) -  \ell_n(\pi',X_{1:n})} \right] \\
 \leqslant 
  c_{\nu}\|\pi-\pi'\|_{\mathsf{tv}}^{2}\log^2 \left(\|\pi-\pi'\|_{\mathsf{tv}}\right) n\eqsp.
\end{equation}
\end{proposition}
\begin{proof}
The proof is postponed to Section~\ref{sec:proof:kl}.
\end{proof}
To prove the existence of exponentially consistent tests, required in \cite{rousseau:2016} for example, we use the concentration inequality obained in Proposition \ref{prop:suppi}. Following \cite[Proposition~1]{douc:moulines:ryden:2004} (see also \cite{diel:lecorff:lerasle:2018}), it is established in Lemma~\ref{lem:likelihoodcontrol} that, under H\ref{assum:lowerbound}, for all $\pi\in\Pi$, there exists a function $\ell_{\pi}$ such that $\Pstar$-a.s.  and in $\mathrm{L}^1(\Pstar)$,
\begin{equation}\label{eq:ApproxRisk}
\frac{1}{n}\ell_n(\pi,X_{1:n}) \limit{n} \kullback_{\bayes}(\pi) = \E_{\bayes}\left[\ell_{\pi}(X)\right]\eqsp. 
\end{equation} 
The quantity $\ell_{\pi}(X)$ is the $\Pstar$-a.s. limit of the sequence $(\log \bP_{\pi}(X_0|X_{-\ell:-1}))_{\ell\geqslant 1}$. For all subset $\widetilde\Pi\subset\Pi$  compact for the topology induced by $\norm{\cdot}_{\mathrm{tv}}$, let $\mathsf{N}(\widetilde{\Pi},\norm{\cdot}_{\mathsf{tv}},\varepsilon)$ be the minimal number of balls of radius $\varepsilon$ necessary to cover $\widetilde\Pi$. 
For all $\pi,\pi'\in\Pi$ and all $n\geqslant 1$, define
\begin{equation}
\label{eq:def:Deltapipi'}
\Delta^n_{\pi,\pi'}(X_{1:n}) = \frac{1}{n}\ell_n(\pi',X_{1:n})-\kullback_{\pi}(\pi')\eqsp.
\end{equation}
\begin{proposition}
\label{prop:suppi}
Assume that H\ref{assum:lowerbound} holds. Then, there exists a constant $c_\nu>0$ such that for all compact subset $\widetilde\Pi\subset\Pi$ for the topology induced by $\norm{\cdot}_{\mathrm{tv}}$, all $\pi\in\widetilde{\Pi}$, all $\varepsilon>0$ and all $t>0$, 
\[
\bP_{\pi}\left(\underset{\pi',\pi''\in\widetilde{\Pi}}{\sup}\left|\Delta^n_{\pi,\pi'}(X_{1:n})-\Delta^n_{\pi,\pi''}(X_{1:n})\right|\geqslant c_\nu\left(\varepsilon+\frac{\sqrt{\log 2\mathsf{N}(\widetilde{\Pi},\norm{\cdot}_{\mathsf{tv}},\varepsilon)}+t}{\sqrt{n}}\right)\right)\leqslant \mathrm{e}^{-t^2}\eqsp.
\]
\end{proposition}
\begin{proof}
The proof is postponed to Section~\ref{sec:proof:deltapi}.
\end{proof}
For all $\varepsilon>0$, define
\begin{equation}
\label{eq:Sstar}
\mathsf{S}_{\star}(\varepsilon) = \{ \pi :  \lVert \pi - \bayes \rVert_{\mathrm{tv}} \leqslant \varepsilon \}\eqsp.
\end{equation}
Consider the following assumption.
\begin{hypH}-\,($c_1,c_2$)
\label{assum:entropy}
There exist a sequence $(\Pi_n)_{n\geqslant 1}$ of subsets of $\Pi$ and two sequences $(\varepsilon_n)_{n\geqslant 1}$ and $(\tilde{\varepsilon}_n)_{n\geqslant 1}$ such that $(n\tilde{\varepsilon}_n^2)^{-1}\log^2(\tilde{\varepsilon}_n) = o(1)$, 
$\tilde{\varepsilon}_n \leq \varepsilon_n $  satisfying
\begin{gather*}
\log \mathsf{N}(\Pi_n,\norm{\cdot}_{\mathrm{tv}},\varepsilon_n) \leq n\varepsilon_n^2\eqsp,\\
\mu\left(\mathsf{S}_{\star}(\tilde{\varepsilon}_n)\right) \geqslant \mathrm{e}^{-c_1n\tilde{\varepsilon}_n^2}\quad\mbox{and}\quad \mu\left(\Pi_n^c\right) \leqslant \mathrm{e}^{-c_2n\tilde{\varepsilon}_n^2}\eqsp. 
\end{gather*}
\end{hypH}
As in \cite{diel:lecorff:lerasle:2018}, using that for all $\pi\in\Pi$, $\pi \in \mathrm{Argmax}\,\kullback_{\pi}$, define the risk function on $\Pi$ as
\[
R_{\bayes}: \pi \mapsto \kullback_{\bayes}(\bayes)-\kullback_{\bayes}(\pi)
\]
and for all $\varepsilon>0$ and all constant $\bar c >0$, 
\[
\mathsf{B}_{\bar c,\star}(\varepsilon) = \left\{\pi\in\Pi\;:\; R_{\bayes}(\pi)\leqslant\bar c\varepsilon\right\}\eqsp.
\]
Under assumption H\ref{assum:entropy}, posterior concentraction rates are derived with respect to the risk function $R_{\bayes}$. 
\begin{theorem}
\label{th:thetheorem}
Assume that H\ref{assum:lowerbound} and H\ref{assum:entropy}\,-\,($c_1,c_2$) hold for constants $c_1,c_2>0$ such that  $c_2>c_{\nu}+c_1$  where $c_{\nu}$ is defined in Proposition~\ref{prop:controle_KL_et_var_ln}.
For any positive sequence $(\alpha_n)_{n\geqslant 1}$ such that $\alpha_n^{-1} = O(\mathrm{e}^{c_{\alpha}n\tilde{\varepsilon}_n^2})$ with $0<c_{\alpha}<c_2-c_{\nu}-c_1 $ and any $\bar c>2(1+\sqrt{2})\left(1+\sqrt{ c_2} \right) c_\nu $,
\[ 
\bP_{\bayes}\left(\mu\left(\mathsf{B}^c_{\bar c, \star}(\varepsilon_n) \middle|X_{1:n}\right)>\alpha_n\right) = o(1)\eqsp.
\]
\end{theorem}
\begin{proof}
The proof is postponed to Section~\ref{sec:proof:th}.
\end{proof}
Theorem~\ref{th:thetheorem} establishes convergence rates of the posterior distribution for the Kullback-Leibler risk $R_{\bayes}$. 
The Kullback-Leibler divergence between two probability distributions usually scales as the square norm between these probabilities, so that  a scaling in $\varepsilon_n^2$ could have been expected in Theorem~\ref{th:thetheorem}. 
The reason is that we use the approach of \cite{diel:lecorff:lerasle:2018} that builds on the analysis of empirical risk minimizers by Vapnik \cite{MR1641250} to bound the risk which results in ``slow" rates of convergence of the estimators. 
``Fast" rates (see \cite{MR2051002} for a discussion on fast and slow rates in learning) in $\varepsilon_n^2$ could be derived under margin type assumptions \cite{MR1765618, MR2829871}.
Proving that such margin conditions are satisfied in our framework is an interesting question that goes beyond the scope of this paper.

\subsection{Examples of prior distributions}
This section presents classical examples of prior distributions satisfying the assumptions of Theorem~\ref{th:thetheorem}.
\paragraph{Dirichlet process mixtures of Gaussian distributions}
Dirichlet process mixtures of Gaussian distributions are commonly used to model densities on $\rset$. In \cite{shen:tokdar:ghosal:2013} and \cite{kruijer:rousseau:vandervaart:2010}, this prior distribution is used to obtain adaptive posterior concentration rates for i.i.d. observations. 
Let $G$ denote a probability measure on $\rset$, with positive density on $\rset$ with respect to the Lebesgue measure and let $\Pi_{\sigma}$ denote a probability measure on $(0,+\infty)$. These measures are assumed to satisfy the following tail assumptions.
\begin{hypH}-($\kappa$)
\label{assum:Tail}
There exist $(b_i)_{i\in\{0,1\}},x_0>0$ such that
\begin{align*}
\forall x>x_0,&\qquad 1-G([ -x,x])\leqslant b_0e^{-b_1x^{b_1}}\eqsp,\\
\forall x<x_0^{-1},&\qquad \Pi_{\sigma}\left([0, x]\right)\leqslant b_0e^{-b_1 x^{-b_1}}\eqsp,\\
\forall x>x_0,&\qquad \Pi_{\sigma}\left([x,+\infty)\right)\leqslant b_0x^{-b_1}\eqsp,\\
\forall s> 0,\ \forall t\in (0,1),&\qquad \Pi_{\sigma}\left((s,s(1+t))\right)\geqslant b_1 s^{-b_0}t^{b_0}e^{-b_0s^{-\kappa/2}}\eqsp. 
\end{align*}
\end{hypH}
Let $a>0$, $\text{DP}(aG)$ be the Dirichlet process with base measure $aG$,  $\phi_{\sigma}$ be the Gaussian density function with mean $0$ and variance $\sigma$ on $\rset$ and assume that under the prior distribution
\begin{equation}\label{def:PriorGauss}
\pi=\int \phi_{\sigma}(.-\mu)\rmd P(\mu),\qquad (P,\sigma)\sim \text{DP}(aG)\otimes\Pi_{\sigma}\eqsp.
\end{equation}
Let $\beta>0$, $\tau\geqslant 0$, $L$ a non-negative function on $\mathbb{R}$ and let $\mathcal{C}(\beta,L,\tau)$ be the class of locally-H\"older functions with regularity $\beta$:
\[
\mathcal{C}(\beta,L,\tau) = \left\{f:\rset\to \mathbb{R}\eqsp :\eqsp \forall(x,y)\in\rset^2\eqsp, \left|f^{(\lfloor \beta\rfloor)}(x)-f^{(\lfloor \beta\rfloor)}(y) \right| \leqslant L(y)e^{\tau|x-y|^2} |x-y|^{\beta-\lfloor \beta\rfloor} \right\}\eqsp,
\]
where $\lfloor \beta\rfloor$ denotes the largest integer smaller than $\beta$ and, for any integer $k$, $f^{(k)}$ denotes the $k$-th derivative of $f$.
Consider also the following assumption.
\begin{hypH}
\label{assum:Loc-holder}-($\beta,L$)
There exist positive constants $\delta$, $x_m$, $b_4$, $b_5$ and $\tau$  such that
\begin{align*}
\forall k \leqslant  \lfloor \beta \rfloor\eqsp, &\qquad \int \left( \frac{\lvert \bayes^{(k)}(x) \rvert}{\bayes(x)} \right)^{(2\beta+\delta)/k} \bayes (\rmd x) < \infty\eqsp,
\\
&\qquad \int \left( \frac{L(x)}{\bayes(x)} \right)^{(2\beta+\delta)/\beta} \bayes (\rmd x) < \infty\eqsp,
\\
\forall \lvert x \rvert > x_m\eqsp, &\qquad \bayes(x) \leq b_4 e^{-b_5 \lvert x \rvert^\tau}\eqsp.
\end{align*}
\end{hypH}
\begin{corollary}
\label{cor:rate:gauss} Let $\bayes\in \mathcal{C}(\beta,L,\tau)$.
Assume that H\ref{assum:Tail}-($\kappa$) and H\ref{assum:Loc-holder}-($\beta,L$) hold with $\kappa>0$ and let $d^*=\max(1,\kappa)$.
Then, there exist $c_1,c_2$ such that the prior distribution $\mu$ defined in \eqref{def:PriorGauss} satisfies H\ref{assum:entropy}-\,($c_1,c_2$) for every positive constants $c_1$ and $c_2$ with, for all $n\geqslant 1$, all $t>t_0>(d^*(1+\tau^{-1}+\beta^{-1})+1)/(2+d^*/\beta)$,
\[ \varepsilon_n \propto n^{-\beta/(2\beta+d^*)}(\log n)^{t}
,\qquad \tilde{\varepsilon}_n \propto n^{-\beta/(2\beta+d^*)}(\log n)^{t_0}\eqsp.
\]
\end{corollary}

\begin{proof}
For all $n\geqslant 1$,  let $\alpha_n = \alpha_0 n^{2/(2\beta+d^*)}(\log n)^{-2t/\beta}$ with $\alpha_0>0$ a sufficiently large constant. By \cite[Theorem 4]{shen:tokdar:ghosal:2013}, 
 \[
\mu\left(\left\{\pi\;:\; \mathrm{KL}(\bayes,\pi)\leqslant \alpha_n^{-\beta}\;,\;V_2(\bayes,\pi)\leqslant \alpha_n^{-\beta}\right\}\right)\geqslant \mathrm{e}^{-c \alpha_n^{1/2}(\log \alpha_n)^{(2+\beta^{-1})t}}\eqsp,
\]
 where for all $p>0$,
\begin{equation}\label{def:KLVp}
 \mathrm{KL}(\bayes,\pi) = \int \bayes(u) \log\left(\frac{\bayes(u)}{\pi(u)}\right)\rmd u \;\;\mbox{and}\;\;\mathrm{V}_p(\bayes,\pi) = \int \bayes(u) \left|\log\left(\frac{\bayes(u)}{\pi(u)}\right)\right|^p\rmd u \eqsp.
\end{equation}
Then, by Pinsker's inequality, $\|\bayes-\pi\|_{\mathrm{tv}}\leqslant \sqrt{2}\,\mathrm{KL}(\bayes,\pi)^{1/2}$ and therefore
\[
\mu\left(S_{\star}(\sqrt{2}\alpha_n^{-\beta/2})\right)\geqslant \mathrm{e}^{-c \alpha_n^{1/2}(\log \alpha_n)^{(2+\beta^{-1})t}}\geqslant \mathrm{e}^{-cn\tilde\varepsilon_n^2}\eqsp,
\]
where $S_{\star}$ is defined by \eqref{eq:Sstar}. This concludes the proof using \cite[Theorem 5]{shen:tokdar:ghosal:2013}.
\end{proof}

\paragraph{Dirichlet process mixtures of Beta distributions}
Following \cite{rousseau:2010}, the posterior concentration rate established in Theorem~\ref{th:thetheorem} may be investigated in the case of mixtures of Beta distributions. For all $\alpha>0$ and all $\varepsilon\in(0,1)$, the probability density function of the Beta distribution with mean $\varepsilon$ and scale parameter $\alpha$ is written:
\[
g_{\alpha,\varepsilon}:x\mapsto \frac{x^{a_{\alpha,\varepsilon}-1}(1-x)^{b_{\alpha,\varepsilon}-1}}{B(a_{\alpha,\varepsilon},b_{\alpha,\varepsilon})}\eqsp,
\]
where $a_{\alpha,\varepsilon} = \alpha/(1-\varepsilon) $, $b_{\alpha,\varepsilon} = \alpha/\varepsilon$ and for all $a,b>0$, $B(a,b) = \Gamma(a)\Gamma(b)/\Gamma(a+b)$, with $\Gamma$ the Gamma function. For any probability distribution $P$ on $(0,1)$ and any $\alpha>0$, define the continuous mixture $g_{\alpha,P}$ by 
\begin{equation}
\label{eq:cont:mixt}
g_{\alpha,P} = \int g_{\alpha,\varepsilon}P(\rmd \varepsilon)\eqsp. 
\end{equation}
\cite{rousseau:2010} introduces the following prior distribution on density functions $\pi=g_{\alpha,P}$ on $(0,1)$ with a random mixture distribution $P$:
\begin{equation}
\label{eq:prior:beta}
(P,\alpha) \sim DP(\nu)\otimes \pi_{\alpha}\eqsp,
\end{equation}
where $\pi = g_{\alpha,P}$ and $P=\sum_{j=1}^k p_j\delta_{\varepsilon_j}$. As in \cite{rousseau:2010}, assume that:
\begin{enumerate}[-]
\item $\eta$ is a probability distribution on $(0,1)$ such that there exist $\underline{\eta}\eqsp, \overline\eta>0$ and $T\geqslant 1$ satisfying, for all $x\in(0,1)$,
\[
0<\underline{\eta} x^{T} (1-x)^{T} \leqslant  \eta(x)\leqslant \overline{\eta} \eqsp;
\] 
\item the density $\pi_{\alpha}$ has support $[n^t,+\infty)$ for some $0<t<1$ and for all $b_1>0$, there exist $c,c_1,c_2,c_3>0$ such that for all $\alpha_n$ satisfying $\alpha_n n^{-t}\to + \infty$,
\[
\pi_{\alpha}\left(c_1\alpha_n<\alpha<c_2\alpha_n\right)\geqslant c \mathrm{e}^{-b_1\alpha_n^{1/2}}\quad \mbox{and}\quad \pi_{\alpha}\left(c_3\alpha_n<\alpha\right)\leqslant c \mathrm{e}^{-b_1\alpha_n^{1/2}} \eqsp.
\]
\end{enumerate}
Let $\beta,L>0$ and $\mathcal{H}(\beta,L)$ be the class of Hölder functions with regularity $\beta$:
\[
\mathcal{H}(\beta,L) = \left\{f:(0,1)\to \mathbb{R}\eqsp :\eqsp \forall(x,y)\in(0,1)^2\eqsp, \left|f^{(\lfloor \beta\rfloor)}(x)-f^{(\lfloor \beta\rfloor)}(y) \right| \leqslant L \left|x-y\right|^{\beta-\lfloor \beta\rfloor} \right\}\eqsp.
\]
Consider also the following assumption.
\begin{hypH}-($\beta$)
\label{assum:holder}
For all $x\in (0,1)$ $\bayes(x)>0$ and there exist $\beta,L>0$ and $0\leqslant k_0,k_1<\beta$ such that
\[
\bayes \in \mathcal{H}(\beta,L)\eqsp,\quad\bayes^{(k_0)}(0)>0\eqsp,\quad \bayes^{(k_1)}(1)<0\eqsp.
\]
\end{hypH}
\begin{corollary}
\label{cor:rate:beta}
Assume that H\ref{assum:holder}-($\beta$) holds with $0<\beta\leqslant t-1/2$. Then, there exist constants $c_1,c_2$ such that the prior distribution \eqref{eq:prior:beta} satisfies H\ref{assum:entropy}-\,($c_1,c_2$) with, for  some sufficienty large constant $c_0$ and all $n\geqslant 1$,
\[
\varepsilon_n = c_0n^{-\beta/(2\beta+1)}(\log n)^{5\beta/(2\beta+1)}\eqsp, \quad \tilde\varepsilon_n =  n^{-\beta/(2\beta+1)}(\log n)^{5\beta/(4\beta+2)}\eqsp.
\]
\end{corollary}
\begin{proof}
In \cite{rousseau:2010}, the author analyzed the approximation of the probability density $\bayes$ on $(0,1)$ by a continuous mixture of the form \eqref{eq:cont:mixt}. When $\bayes$ is Hölder with regularity $\beta>0$, it is shown in \cite[Theorem~3.1]{rousseau:2010} that for all $\alpha>0$ and all $p>1$, there exists a probability density $\pi_{\alpha}$ on $(0,1)$ such that $\mathrm{KL}(\bayes,g_{\alpha,\pi_{\alpha}})\leqslant c\alpha^{-\beta}$, $\mathrm{V}_p(\bayes,g_{\alpha,\pi_{\alpha}})\leqslant c\alpha ^{-\beta}$ and $\|\bayes - g_{\alpha,\pi_{\alpha}}\|_{\infty}\leqslant \alpha^{-\beta/2}$ where $\mathrm{KL}$ and $\mathrm{V}_p$ are defined in \eqref{def:KLVp}.
If $0<\beta\leqslant 2$ we may choose $\pi_{\alpha} = \bayes$ and, if $\beta > 2$, $\pi_{\alpha}$ must be different from $\pi$ to obtain the optimal approximation. Therefore, if $\bayes$ is Hölder with regularity $\beta>0$, $\bayes$ can be sharply approximated by a continuous mixture of Beta distributions of the form \eqref{eq:cont:mixt}. Then, \cite{rousseau:2010} introduced a discrete mixture approximation of $g_{\alpha,\pi_{\alpha}}$ of the form $g_{\alpha,P_{\alpha}}$ where $P_{\alpha}$ is a probability distribution with finite support.
Using this distribution, it is shown in \cite[page 171]{rousseau:2010} that
\[
\mu\left(\left\{\pi\;:\; \mathrm{KL}(\bayes,\pi)\leqslant \alpha_n^{-\beta}\;,\;V_2(\bayes,\pi)\leqslant \alpha_n^{-\beta}\right\}\right)\geqslant \mathrm{e}^{-c \alpha_n^{1/2}(\log \alpha_n)^{5/2}}\eqsp,
\]
for some positive constant $c$.
 Then, by Pinsker inequality, $\|\bayes-\pi\|_{\mathrm{tv}}\leqslant \sqrt{2}\,\mathrm{KL}(\bayes,\pi)^{1/2}$ and
\[
\mu\left(S_{\star}(\sqrt{2}\alpha_n^{-\beta/2})\right)\geqslant \mathrm{e}^{-cN_0 \alpha_n^{1/2}(\log \alpha_n)^{5/2}}\geqslant \mathrm{e}^{-cn\tilde{\varepsilon}_n^2}\eqsp,
\]
where $\tilde{\varepsilon}_n = \alpha_n^{-\beta/2} = \varepsilon_0 n^{-\beta/(2\beta+1)}(\log n)^{5\beta/(4\beta+2)}$ and where $S_{\star}$ is defined by \eqref{eq:Sstar}.
On the other hand, \cite[proof of theorem 2.2, pages 171 to 173]{rousseau:2010} introduces a sequence of subsets $(\Pi_n)_{n\geqslant 1}$ such that,
\[
\mu\left(\Pi^c_n\right)\leqslant C \alpha_n^\beta \mathrm{e}^{-\alpha \sqrt{\alpha_n} \log(\alpha_n)^{5/2}}\leqslant \mathrm{e}^{-cn \tilde\varepsilon_n^2}
\]
and with a log-entropy of $\Pi_n$ in total variation distance lower than $n\varepsilon^2_n$. 
\end{proof}

\section{Numerical experiments}
\label{sec:exp}
\subsection{Block Gibbs sampling algorithm}
\label{sec:alg}
Following \cite{yau:papaspiliopoulos:roberts:holmes:2011}, a Markov Chain Monte Carlo (MCMC) method can be introduced to solve the Bayesian nonparametric problem introduced in this paper, i.e. to sample from the posterior distribution given $X_{1:n}$. The unknown distribution $\pi$ is specified as a mixture model in which some probability density $\varphi_z$ is mixed with respect to a discrete probability measure $P$. Defining the discrete probability measure $P$ as in \cite{yau:papaspiliopoulos:roberts:holmes:2011}, the mixture of Dirichlet processes with base measure $\alpha Q$, where $\alpha$ is a positive constant and $Q$ a probability distribution, is given by:
\begin{align}
(\vartheta_j)_{j\geqslant 1} &\underset{\mbox{i.i.d.}}{\sim} \mathrm{Beta}(1,\alpha)\eqsp,\\
\omega_1 = \vartheta_1 \quad&\mbox{and}\quad \mbox{for}\eqsp\eqsp j\geqslant 2\eqsp,\eqsp\eqsp \omega_j = \vartheta_j\prod_{i=1}^{j-1}(1-\vartheta_i)\eqsp,\\
(z_j)_{j\geqslant 1} &\underset{\mbox{i.i.d.}}{\sim} Q\eqsp,\\
(\kappa_i,u_i)_{1\leqslant i \leqslant n+1} &\underset{\mbox{i.i.d.}}{\sim} \sum_{j\geqslant 1}\un{u_i<\omega_j}\delta_j(\kappa_i)\eqsp,\label{eq:prior:ku}\\
V_i &\sim \varphi_{z_{\kappa_i}}\quad \mbox{for}\eqsp\eqsp 1\leqslant i\leqslant n+1\eqsp,
\end{align}
The Dirichlet process prior is defined by the stick-breaking weights $(\vartheta_j)_{j\geqslant 1}$, the mixture parameters $(z_j)_{j\geqslant 1}$ the allocation variables $(\kappa_i)_{1\leqslant i \leqslant n+1}$ ans some auxiliary variables $(u_i)_{1\leqslant i \leqslant n+1}$. This representation with auxiliary variables is introduced in \cite{walker:2007} and used in \cite{yau:papaspiliopoulos:roberts:holmes:2011}. For all $1\leqslant i \leqslant n+1$, integrating out the random variable $u_i$ on $(0,1)$ leads to the usual marginal distribution for $\kappa_i$:
\[
\kappa_i\underset{\mbox{i.i.d.}}{\sim} \sum_{j\geqslant 1}\omega_j\delta_j\eqsp,
\]
and moreover integrating out the random variable $\kappa_i$, marginally
\[
V_i \sim \sum_{j\geqslant 1}\omega_j \varphi_{z_{j}} = 
\int \varphi_{z}(\cdot) P(\rmd z) \eqsp,
\]
where $P=\sum_{j\geqslant 1}\omega_j\delta_j $ is distributed from a Dirichlet process with base measure $\alpha Q$.
The auxiliary variables greatly simplifies the expression of the posterior distributions of $\kappa_i$ as detailed in \ref{it:step:kappa} below. Sampling from the joint posterior distribution of $(V_{1:n+1},u_{1:n+1},\kappa_{1:n+1},z,\vartheta)$ is then performed by block Gibbs sampling according to the following steps, where $z=(z_j)_{j\geqslant 1}$ and $\vartheta=(\vartheta_j)_{j\geqslant 1}$.
\begin{enumerate}[(i)]
\item \label{it:step:smooth}The hidden states $V_{1:n+1}$ are sampled according to their posterior distribution given the random variables  $(X_{1:n},u_{1:n+1},z,\vartheta)$ with a collapsed step. Note that by integrating out $\kappa_{1:n+1}$, the conditional distribution of $(X_{1:n},V_{1:n+1})$ given $(u_{1:n+1},z,\vartheta)$ is given by
\[
p_n(V_{1:n+1},X_{1:n}|u_{1:n+1},z,\vartheta) \propto \prod_{i=1}^{n+1}\left(\sum_{j,\,u_i<\omega_j}\varphi_{z_j}(V_i)\right)\prod_{i=1}^{n}\condlik(X_i,V_i,V_{i+1})\eqsp.
\]
Therefore, the posterior distribution of $V_{1:n+1}$ given the random variables  $(X_{1:n},u_{1:n+1},z,\vartheta)$ is the joint smoothing distribution of $V_{1:n+1}$ given $X_{1:n}$ when $(V_i)_{1\leqslant i \leqslant n+1}$ are independent with $V_i \sim \sum_{j,\,u_i<\omega_j}\varphi_{z_j}$ for all $1\leqslant i \leqslant n+1$. This step cannot be done explicitly as in \cite{yau:papaspiliopoulos:roberts:holmes:2011} where the hidden states are discrete. A Sequential Monte Carlo smoother described below is used instead.
\item \label{it:step:kappa} For all $1\leqslant i\leqslant n+1$, the posterior distribution of $\kappa_i$ given $(X_{1:n},V_{1:n+1},u_{1:n+1},\vartheta,z)$ is given by
\[
\kappa_i | X_{1:n},V_{1:n+1},u_{1:n+1},\vartheta,z \sim \sum_{j,\,\omega_j>u_i}\varphi_{z_{j}}(V_i)\delta_{j}(\kappa_i)\eqsp.
\]
This expression motivates the use of the auxiliary variables $u_{1:n+1}$. In the case where $u_{1:n+1}$ are not introduced in the model, $\kappa_i$ is sampled from its posterior distribution given $(X_{1:n},\kappa_{1:n+1},z,\vartheta)$ which is proportional to $\sum_{j\geqslant 1}\omega_j\varphi_{z_{j}}(V_i)\delta_{j}(\kappa_i)$. Due to the infinite sum, this distribution has an intractable normalizing constant so that sampling $\kappa_i$ without the auxiliary variable $u_i$ is challenging.
\item $(\vartheta,u_{1:n+1})$ are sampled according to their posterior distribution given $(\kappa_{1:n+1},X_{1:n},V_{1:n+1},z)$. As detailed in \cite{yau:papaspiliopoulos:roberts:holmes:2011}, $(\vartheta_j)_{j\geqslant 1}$ are updated according to their distribution given $(\kappa_{1:n+1},X_{1:n},V_{1:n+1},z)$:
\[
\vartheta_j | \kappa_{1:n+1},X_{1:n},V_{1:n+1},z \underset{\mbox{ind.}}{\sim} \mathrm{Be}\left(m_j+1,n+1-\sum_{\ell=1}^{j}m_{\ell}+\alpha\right)\eqsp,
\]
where $m_j = \mathrm{Card}(\{1\leqslant i \leqslant n+1\,;\, k_i = j\})$. Then, the random variables $(u_i)_{1\leqslant i\leqslant n+1}$ are conditionally independent given $(\kappa_{1:n+1},\alpha,X_{1:n},V_{1:n+1},z,\vartheta)$ and such that $u_i \sim U(0,\omega_{k_i})$.
\item $z$ are sampled according to their posterior distribution given $(X_{1:n},\kappa_{1:n},V_{1:n+1},u_{1:n+1},\vartheta)$. This posterior distribution is given, for all $j\geqslant 1$ by 
\[
z_j | X_{1:n},\kappa_{1:n},V_{1:n+1},u_{1:n+1},\vartheta \underset{\mbox{ind.}}{\sim} Q(z_j) \prod_{i=1,\,\kappa_i=j}^{n+1}\varphi_{z_j}(V_i)\eqsp.
\]
\end{enumerate}
Following \cite[Section~4.2.2]{yau:papaspiliopoulos:roberts:holmes:2011}, the probability density function $\varphi_z$ is set as a Gaussian probability density with mean $\mu$ and variance $1/\lambda$ where $z=(\mu,\lambda)$. The distribution $Q$ is chosen as $Q = \mathcal{N}(0,1)\otimes \mathrm{Gamma}(1,1)$, a standard Gaussian distribution for $\mu$ which is independent of the precision parameter $\lambda$ distributed according to a $\mathrm{Gamma}(1,1)$.

\subsubsection*{Sequential Monte Carlo smoother for step \ref{it:step:smooth}}
A Sequential Monte Carlo smoother is used to sample the hidden states $V_{1:n+1}$ according to their posterior distribution given the random variables  $(X_{1:n},u_{1:n},z,\vartheta)$. This is the joint smoothing distribution of $V_{1:n+1}$ given $X_{1:n}$ when marginally $(V_i)_{1\leqslant i \leqslant n+1}$ are independent with $V_i \sim \sum_{j,\,u_i<\omega_j}\varphi_{z_j}$ for all $1\leqslant i\leqslant n+1$. 
\paragraph{Particle filtering}
Particle filtering algorithms are simulation based procedures used to approximate recursively the distributions $\phi^{}_{k}$ of $V^{}_k$ given $X^{}_{1:k-1}$ for $k=1$ to $k=n+1$:
\[
\phi^{}_{k}[h]  = \E_{}^{}\left[h(V^{}_{k})\middle|X_{1:k-1}\right]\eqsp.
\] 
The auxiliary particle filter of  \cite{liu:chen:1998,pitt:shephard:1999} is the most widely used particle filtering method and is a generalization of other approaches such as the algorithms proposed in \cite{gordon:salmond:smith:1993} and \cite{kitagawa:1996}. The auxiliary particle filter produces recursively a set of states $(\xi^{\ell}_k)_{\ell=1}^M$ associated with importance weights $(\omega^{\ell}_k)_{\ell=1}^M$ using importance sampling and resampling steps. At $k = 1$, $(\xi^{\ell}_1)_{\ell=1}^M$ are sampled independently from $\varphi_{z_{\kappa_1}}$ and each particle $\xi^{\ell}_1$ is associated with the standard importance sampling weight $\omega_1^{\ell} = 1/M$. For any bounded and measurable function $h$ defined on $\mathsf{X}$, the expectation $\phi_{1}[h] $ is approximated by a self normalized importance sampling estimator:
\[
\phi^{M}_{1}[h] = M^{-1} \sum_{\ell=1}^M h \left(\xi^{\ell}_1 \right)\eqsp.
\]
Then, using $\{(\xi^{\ell}_{k-1},\omega^{\ell}_{k-1})\}_{\ell=1}^M$,  the auxiliary particle filter of \cite{liu:chen:1998,pitt:shephard:1999} samples pairs $\{(I^{\ell}_k,\xi^{\ell}_{k})\}_{\ell=1}^M$ of superscripts and particles  from the instrumental distribution: 
\[
\pi_{k}(\ell,x) \propto \omega_{k-1}^{\ell} \rho_k (\xi^{\ell}_{k-1}) p_k(\xi^{\ell}_{k-1},x)
\]
defined on $\{1,\ldots,M\} \times \cV$, where $(\rho_k (\xi^{\ell}_{k-1}))_{\ell=1}^M$ are adjustment multiplier weights and $p_k$ is a transition density chosen by the user. In the definition of $\pi_{k}$, the weights $\omega_{k-1}^{\ell} \rho_k (\xi^{\ell}_{k-1})$, $1\leqslant \ell\leqslant M$,  are used to duplicate promising particles or discard particles with low importance weights at the previous time step. 
For $\ell \in \{1,\ldots,M\}$, $\xi^{\ell}_k$ is associated with the following importance weight:
\begin{equation}
\label{eq:weight}
\omega^{\ell}_k = \frac{\left(\sum_{j,\,u_k<\omega_j}\varphi_{z_j}(\xi^{\ell}_k)\right)\condlik(X^{}_{k},\xi_{k-1}^{I_k^{\ell}},\xi_{k}^{\ell} )}{\rho_k(\xi_{k-1}^{\ell}) p_k (\xi_{k-1}^{I_k^{\ell}},\xi^{\ell}_k)}\eqsp.
\end{equation}
\paragraph{Particle smoothing}
The Forward Filtering Backward Simulation (FFBSi) algorithm proposed by \cite{godsill:doucet:west:2004} (see also \cite{delmoral:doucet:singh:2010,douc:garivier:moulines:olsson:2011,dubarry:lecorff:2013} for its convergence properties and \cite{dubarry:lecorff:2011} for its computational efficiency) samples backward trajectories from time $n+1$ to time $1$ among all the $M^{n+2}$ trajectories which can be built using the particles $\xi^{\ell}_k$, with $1\leqslant \ell\leqslant M$ and $1\leqslant k\leqslant n+1$. 
Each trajectory is sampled according to the following steps:
\begin{enumerate}[-]
\item At time $k = n+1$, sample $J^{}_{n+1}\in\{1\ldots,M\}$ with probabilities proportional to $\omega^{\ell}_{n+1}$, for $\ell\in\{1, \ldots, M\}$.
\item 
For $k = n$ to $k=0$, define for all $1\leqslant \ell \leqslant M$,
\begin{equation}
\label{eq:Lambda}
\Lambda_k^{M}(J^{}_{k+1},\ell) = \frac{\omega^{\ell}_k \condlik(X^{}_{k},\xi_{k}^{\ell},\xi_{k+1}^{J_k+1})}{\sum_{j=1}^M \omega^{j}_k \condlik(X^{}_{k},\xi_{k}^{j},\xi_{k+1}^{J_k+1})}\eqsp,\quad 1\leqslant \ell\leqslant M\eqsp.
\end{equation}
Then, $J^{}_k$ is sampled in $\{1\ldots,M\}$ with probabilities $\Lambda_k^{M}(J^{}_{k+1},\ell)$, $1\leqslant \ell \leqslant M$.
\end{enumerate}
These sampling steps produced a trajectory $(\xi^{J_1}_{0},\ldots, \xi^{J_{n+1}}_{n+1})$ approximately distributed according to the target joint smoothing distribution. The computational complexity of the FFBSi algorithm grows with $M^2$ due to the normalizing constant of $\Lambda^{M}_k$, $1\leqslant k\leqslant n$. In the context of this paper, $\condlik$ is upper bounded by 1 and the acceptance rejection mechanism introduced in \cite{douc:garivier:moulines:olsson:2011} to implement the FFBSi algorithm with a complexity which grows linearly with $M$ can be used. In this case, at each time step $1\leqslant k \leqslant n$, the new index $J^{}_k$ is sampled in $\{1,\ldots,M\}$ with probabilities proportional to $(\omega_k^{\ell})_{\ell=1}^M$ and accepted with probability $\condlik(X^{}_{k},\xi_{k}^{J_k}, \xi_{k+1}^{J_{k+1}})$.

\subsection{Application to soccer results}
The algorithm introduced in Section~\ref{sec:alg} is used in this section to estimate the distribution of French soccer teams abilities. Using the outcomes (wins, losses, ties) of the first 30 games during the season 2017-2018, the Expectation Maximization (EM) algorithm for generalized Bradley Terry model with home advantage and ties described in \cite{caron:doucet:2012} is run to estimate the abilities of the 20 teams with $\cX = \{-1,0,1\}$ and the distribution
\begin{equation}
\label{eq:model:exp}
\condlik(1,v_i,v_j) = \frac{\alpha\rme^{v_i}}{\alpha\rme^{v_i}+\theta\rme^{v_j}}\quad\mbox{and}\quad \condlik(0,v_i,v_j) = \frac{(\theta^2-1)\alpha\rme^{v_i}\rme^{v_j}}{(\alpha\rme^{v_i}+\theta\rme^{v_j})(\theta \alpha\rme^{v_i}+\rme^{v_j})}\eqsp.
\end{equation}
The parameters $\alpha$, characterizing the home advantage ($v_i$ is the home team), and $\theta$, associated with ties, are also estimated with the EM algorithm.
The estimates $(\hat v_i)_{1\leqslant i \leqslant 20}$, $\hat\alpha$ and $\hat \theta$ are set as the estimates produced by the algorithm {\em btemhometies}\footnote{http://www.stats.ox.ac.uk/$\sim$caron/code/bayesbt/index.html}, see Figure~\ref{fig:result:EM:Caron:Doucet}. The (supposedly) unknown target density $\bayes$ for the Block Gibbs sampling algorithm is then defined as kernel density estimate of the strengths (rescaled around zero) $(\hat v_i)_{1\leqslant i \leqslant 20}$ obtained with the function {\em ksdensity} of Matlab, see Figure~\ref{fig:result:EM:Caron:Doucet}.

Then $n = 3000$ observations are sampled by first sampling independently $(V_i)_{1\leqslant i \leqslant n+1}$ distributed according to this target density and the observation model \eqref{eq:model:exp} with $\hat \alpha$ and $\hat \theta$.
The initial estimate of the Block Gibbs sampling algorithm is set as a mixture of two Gaussian distributions  $\mathcal{N}(\mu_1,\sigma_1^2)$ and $\mathcal{N}(\mu_2,\sigma_2^2)$ with $\mu_1$ and $\mu_2$ independent with distribution $\mathcal{N}(0,1)$ and $\sigma_1^{-2}$ and $\sigma_2^{-2}$ independent with distribution $\mathrm{Gamma}(1,1)$. The weight  of the first Gaussian distribution is distributed as a $\mathrm{Beta}(1,\alpha)$ random variable with $\alpha = 1$. 

\begin{figure}
\includegraphics[scale = .35]{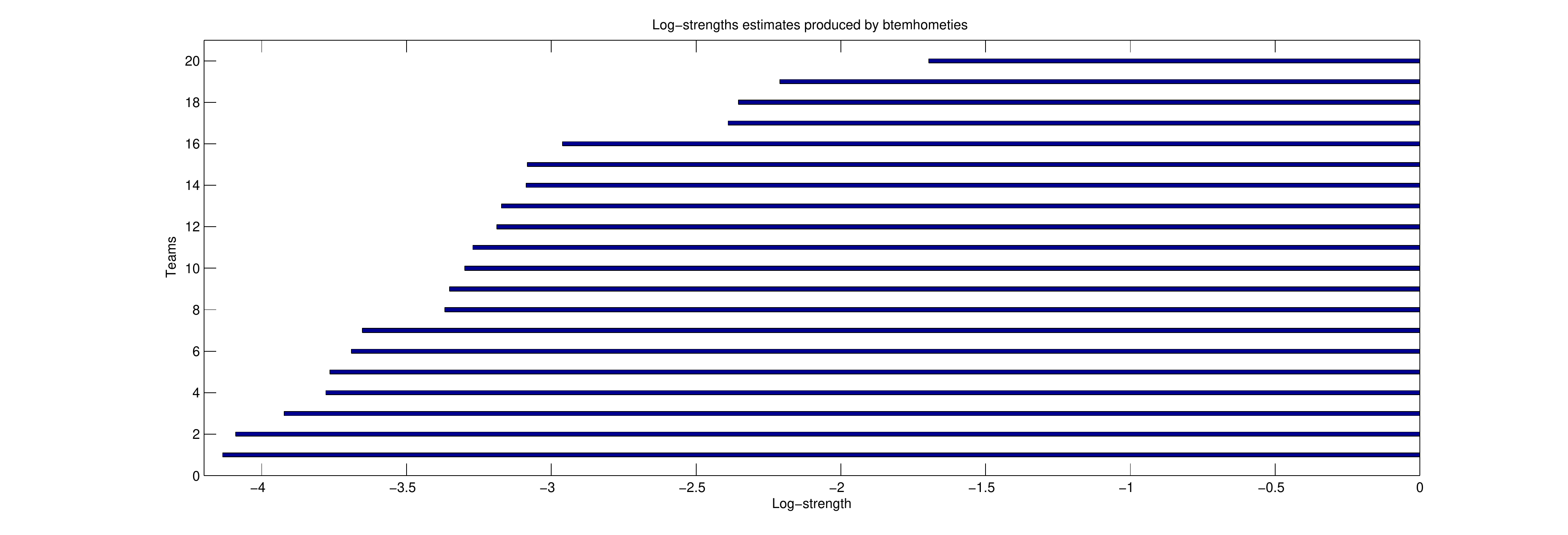}
\\
\includegraphics[scale = .5]{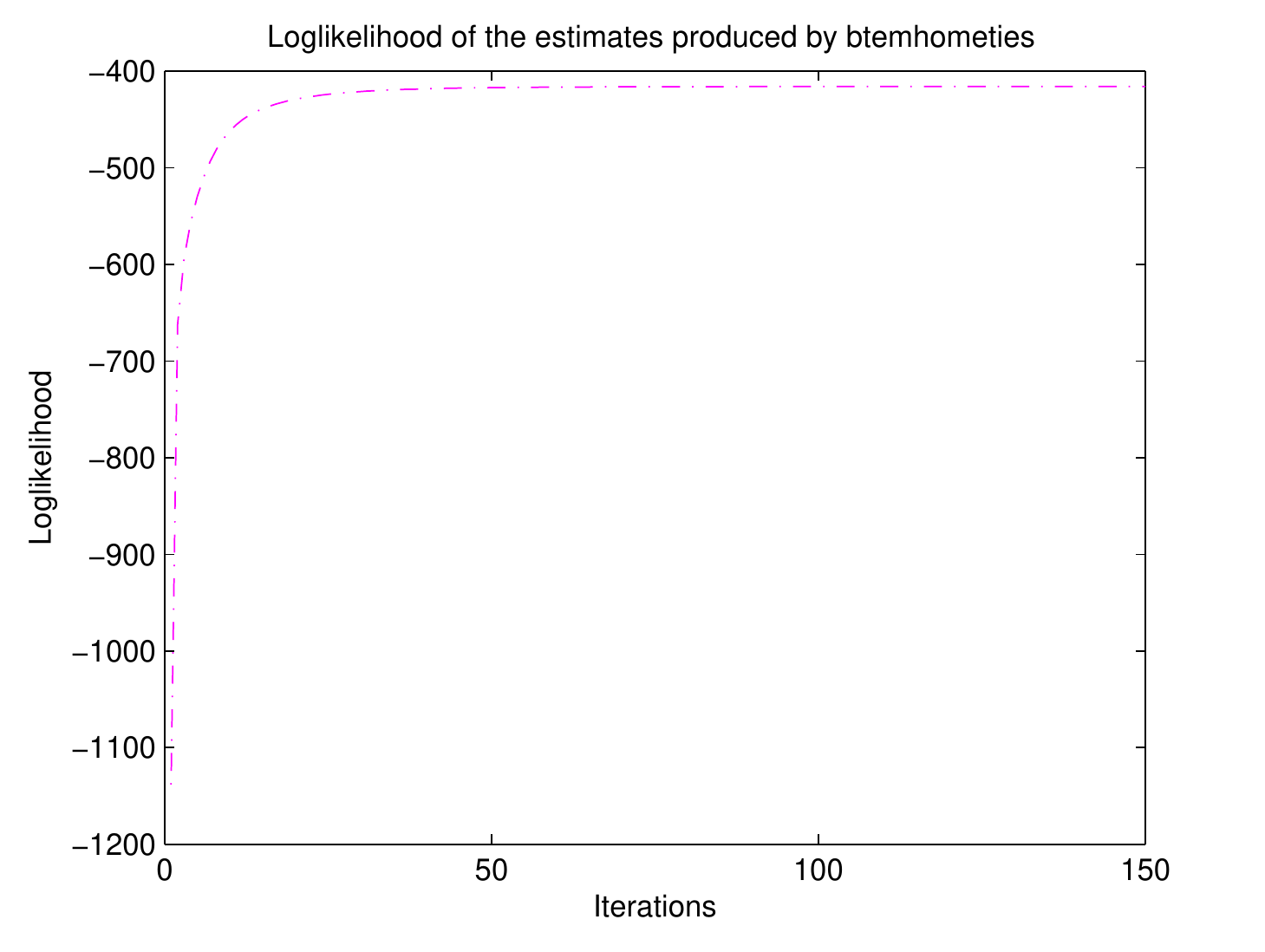}
\includegraphics[scale = .5]{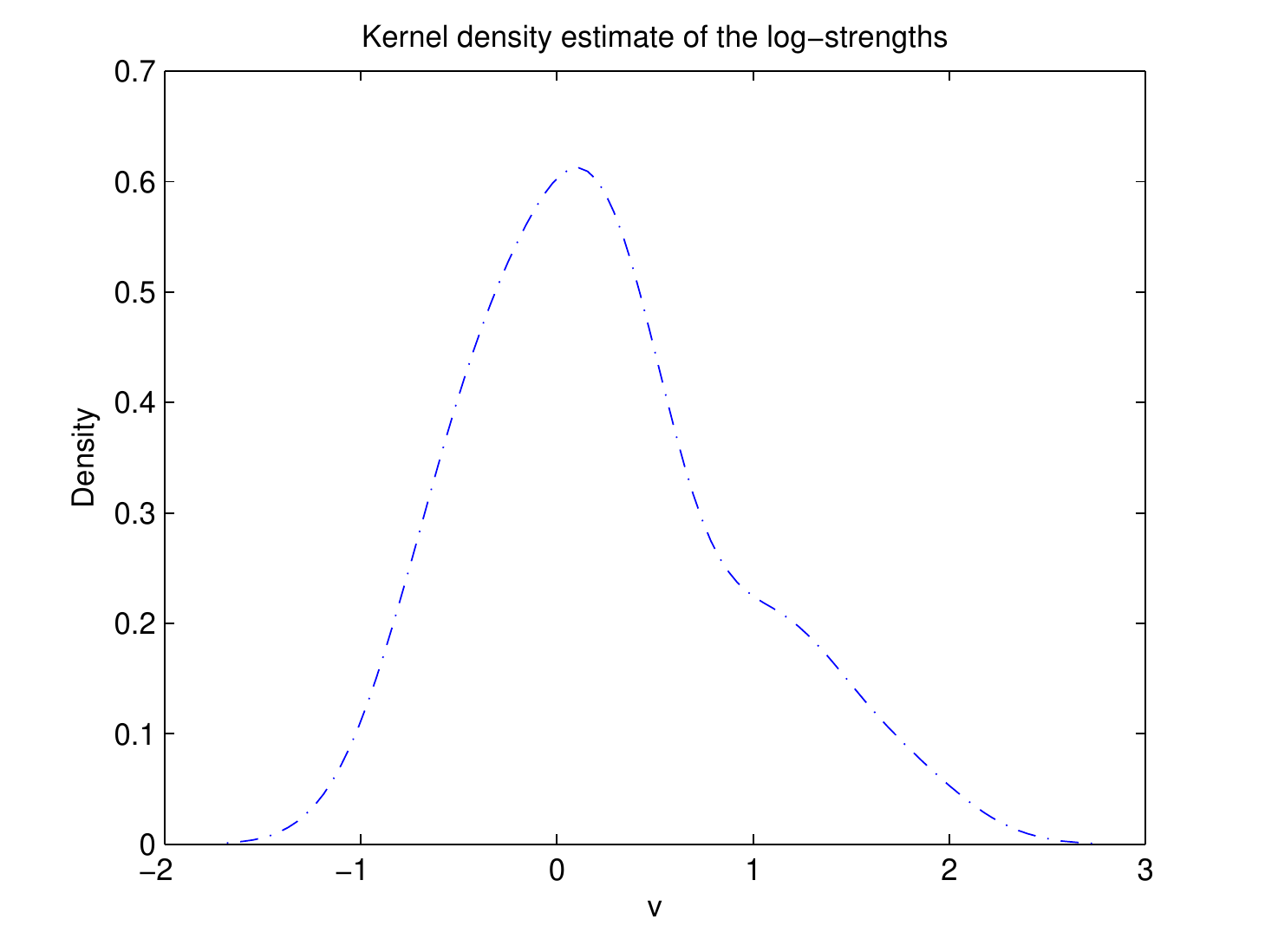}
\caption{Results of the EM algorithm btemhometies with 2017-2018 Ligue 1 results.}
\label{fig:result:EM:Caron:Doucet}
\end{figure}
The Sequential Monte Carlo smoother is set as a standard {\em Bootstrap filter} i.e. at each time step $1\leqslant k\leqslant n+1$, $\rho_k = 1$ and $p_k(x,x') = \sum_{j\eqsp,,u_k<\omega_j}\varphi_{z_j}(x')$.
The number of particles is set to $M=100$ and $6500$ iterations of the Block Gibbs sampling algorithm are performed. The first $5000$ samples are discarded as a burn-in period and the estimate of $\bayes$ is set as the average over the last $1500$ samples  produced by the Block Gibbs Sampler, see Figure~\ref{fig:estimates:BGS}. By \eqref{eq:model:exp}, the law of the observations only depends on $(\mathrm{e}^{V_i})_{1\leqslant i \leqslant n+1}$, therefore $\bayes$ can only be identified up to a translation on $\mathbb{R}$. Such a translation is displayed in Figure~\ref{fig:estimates:BGS} to highlight that the Block Gibbs Sampler recovers mainly the shape of $\bayes$.

In Figure~\ref{fig:estimates:ranking:true} and Figure~\ref{fig:estimates:ranking:dens} 1000 championships are sampled independently using for each run $20$ parameters sampled with the estimated density obtained with the Block Gibbs Sampler to produce boxplots of the scores at the end of the championship. These results are compared to the scores of 1000 championships run with the parameter estimates provided by {\em btemhometies} in Figure~\ref{fig:estimates:ranking:true} and with parameters sampled from the target density in Figure~\ref{fig:estimates:ranking:dens}.

\begin{figure}
\includegraphics[scale = .4]{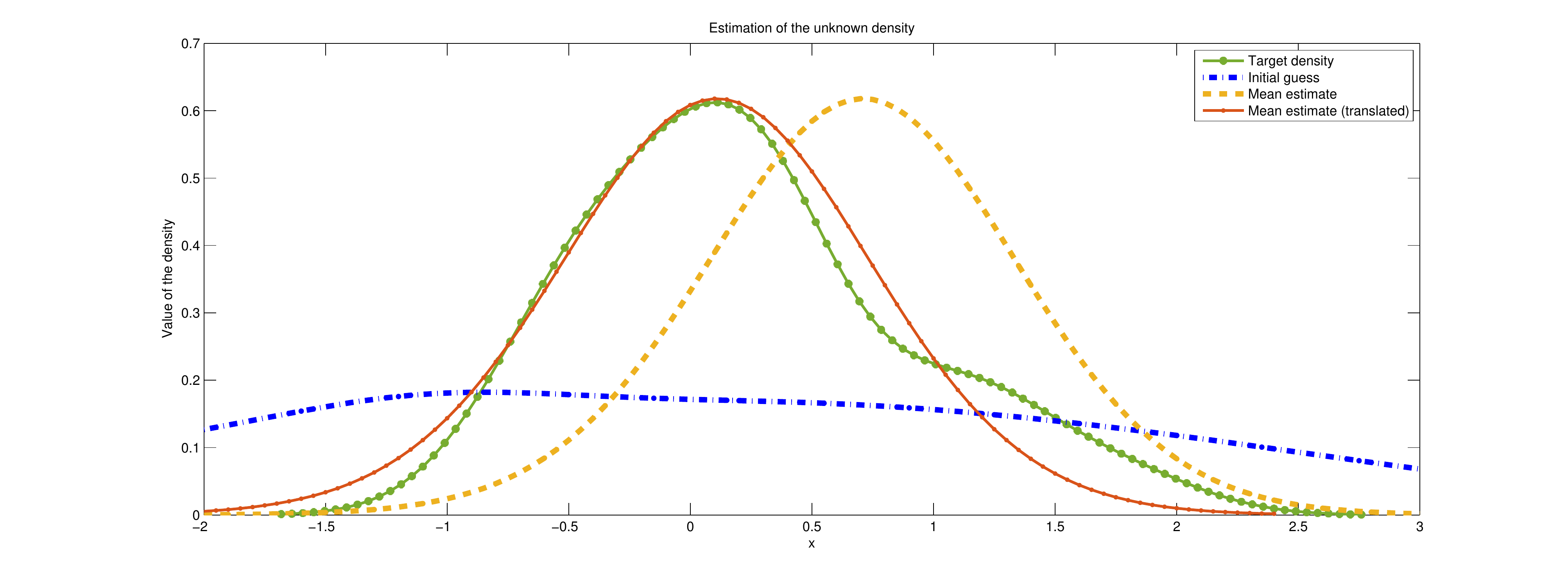}
\caption{Density estimates from the Block Gibbs sampler.}
\label{fig:estimates:BGS}
\end{figure}

\begin{figure}
\includegraphics[scale = .4]{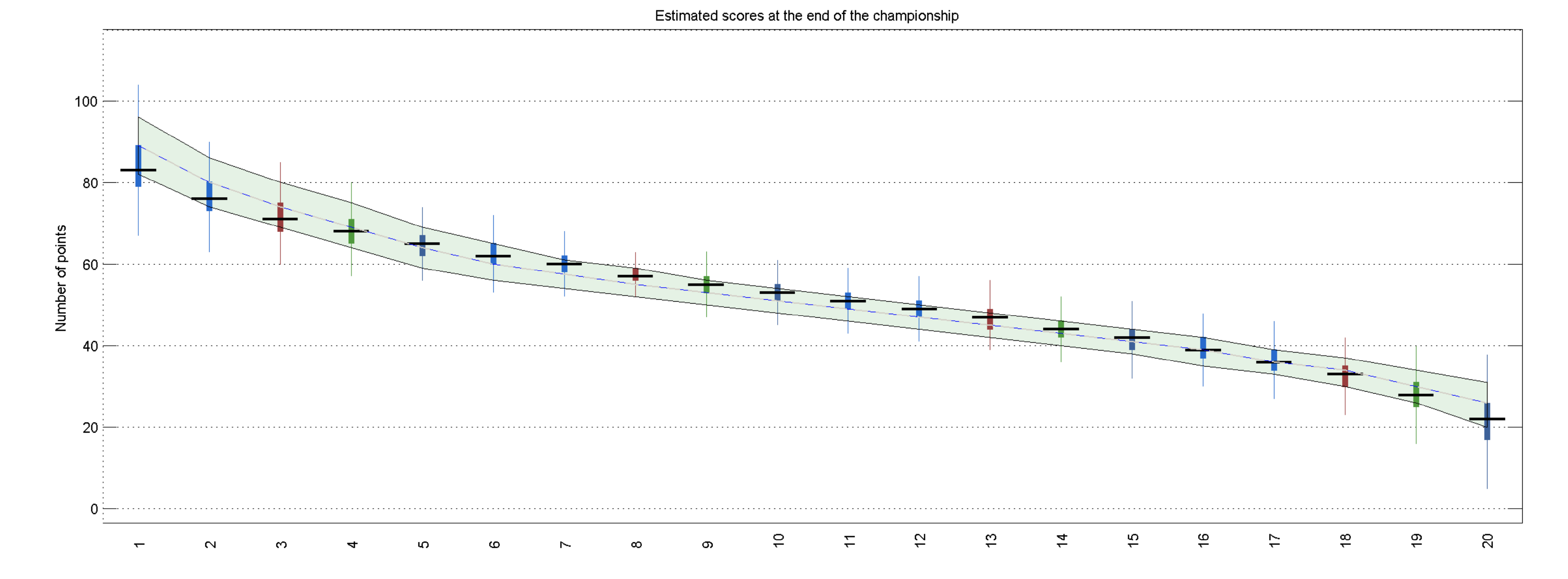}
\caption{Estimated scores at the end of the championship with {\em btemhometies} parameter estimates: median (dotted line) and first and last deciles (grey area).  Boxplots of the scores obtained with the Block Gibbs Sampler.}
\label{fig:estimates:ranking:true}
\end{figure}

\begin{figure}

\includegraphics[scale = .4]{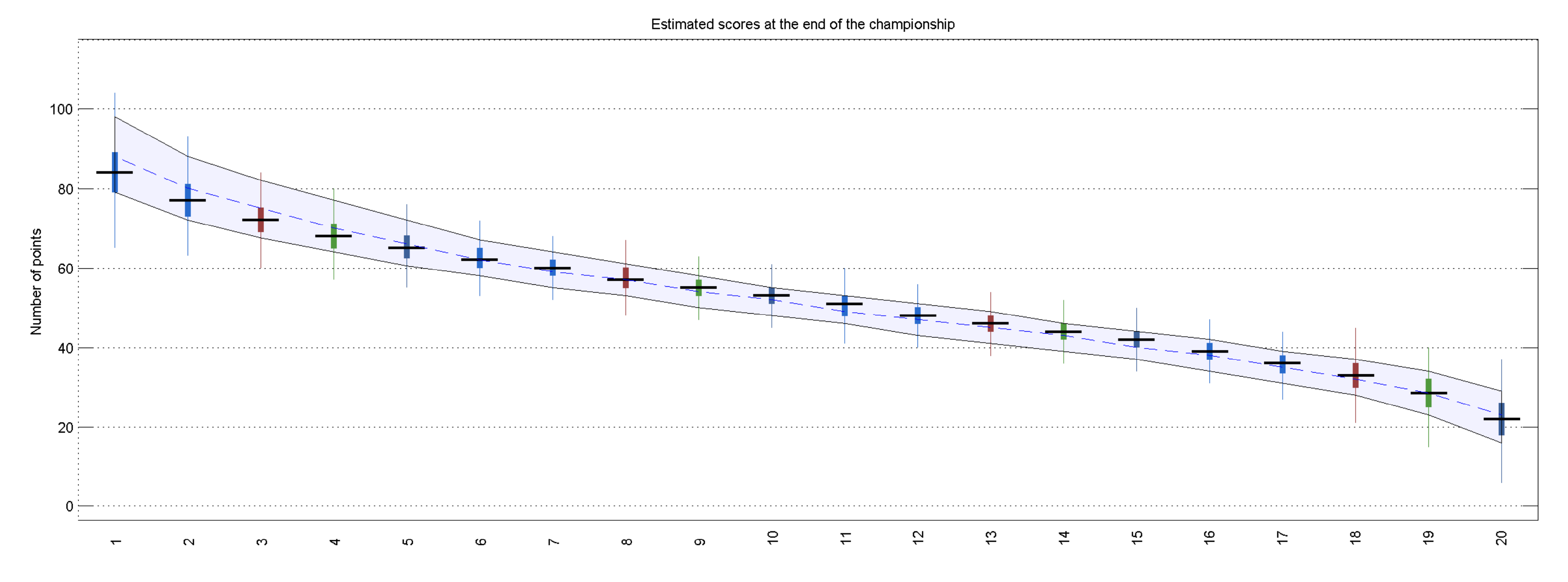}
\caption{Estimated scores at the end of the championship with parameter sampled with the target distribution: median (dotted line) and first and last deciles (grey area).  Boxplots of the scores obtained with the Block Gibbs Sampler.}
\label{fig:estimates:ranking:dens}
\end{figure}

\appendix 

\section{Proof of Theorem~\ref{th:thetheorem}}

\subsection{Proof of Proposition~\ref{prop:controle_KL_et_var_ln}}
\label{sec:proof:kl}
Note first that
\begin{align*}
\E_{\pi}\left[{\log \bP_{\pi}(X_{1:n}) - \log \bP_{\pi'}(X_{1:n})} \right] 
&=\sum_{i=1}^n \E_{\pi}\left[\log\left( \frac{\bP_{\pi}(X_{i}|X_{1:i-1})}{\bP_{\pi'}(X_{i}|X_{1:i-1})} \right) \right]\eqsp,\\
&= 
\sum_{i=1}^n \E_{\pi}\left[\sum_{x_i \in \cX}  \bP_{\pi}(X_i=x_{i}|X_{1:i-1}) \log\left( \frac{\bP_{\pi}(X_i=x_{i}|X_{1:i-1})}{\bP_{\pi'}(X_i=x_{i}|X_{1:i-1})} \right) \right]\eqsp.
\end{align*}
The inner term of the expectation may be upper bounded using that for two probability distributions $p$ and $q$ on $\cX$,
\[
\sum_{x \in \cX} p(x)\log\left(\frac{p(x)}{q(x)}\right)\leqslant \sum_{x \in \cX} p(x) \frac{p(x)-q(x)}{q(x)} = \sum_{x \in \cX} \frac{(p(x)-q(x))^2}{q(x)}\eqsp.
\]
Then, by Lemma~\ref{lem:IncrementsGen},
\begin{align*}
\sum_{x_i \in \cX}  \bP_{\pi}(X_i=x_{i}|X_{1:i-1}) \log\left( \frac{\bP_{\pi}(X_i=x_{i}|X_{1:i-1})}{\bP_{\pi'}(X_i=x_{i}|X_{1:i-1})} \right) &\\
&
\hspace{-3cm}\leqslant \sum_{x_i} \frac{ \left( \bP_{\pi}(X_i=x_{i}|X_{1:i-1}) - \bP_{\pi'}(X_i=x_{i}|X_{1:i-1})\right)^2
}{\bP_{\pi'}(X_i=x_{i}|X_{1:i-1})}\eqsp, \\
&
\hspace{-3cm} \leqslant   4|\cX|\nu^{-5}(2+\nu^{-1})^2\|\pi-\pi'\|_{\mathsf{tv}}^2\eqsp,
\end{align*}
which concludes the proof of \eqref{eq:control_KL}. By \cite[Equation~(32)]{vernet:2015},
\begin{align}
\mathrm{Var}_{\pi} & \left[  {\log \bP_{\pi}(X_{1:n}) - \log \bP_{\pi'}(X_{1:n})} \right] \nonumber\\
 & \leqslant 
4 \sum_{1\leqslant i\leqslant n} \E_{\pi} \left[ Z_i^2 \right] 
+ 4 \sum_{1\leqslant i<j\leqslant n} \E_{\pi} \left[ Z_j^2 \right]^{1/2}
\E_{\pi} \left[\left( \E_{\pi} \left[ Z_j | X_{1:i-1} \right] -\E_{\pi} \left[ Z_j \right] \right)^2\right]^{1/2}\eqsp,\label{eq:var}
\end{align}
where, for all $i\geqslant 1$, $Z_i= \log\bP_{\pi}(X_{i}|X_{1:i-1}) - \log\bP_{\pi'}(X_{i}|X_{1:i-1})$. Following \cite{douc:matias:2001}, \eqref{eq:var} can be upper bounded using the extended Markov chain defined, for all $1\leqslant i\leqslant n$, by
\begin{equation}
\label{eq:ext:mc}
R_i = \left(V_{i+1},X_i,\eta^{\pi}_{i},\eta^{\bayes}_{i}\right)\eqsp,
\end{equation} 
where, for all $i\geqslant 1$, $\eta^{\pi}_{i}$ is the predictive filter at time $i$ when the law of the hidden states is $\pi$: $\eta^{\pi}_{1} = \pi$ and for $i\geqslant 2$,
\[
\eta^{\pi}_{i}:A\mapsto \bP_{\pi}\left(V_i\in A\middle |X_{1:i-1}\right)\eqsp.
\]
Define the function $\mathsf{h}$, for all $r=(v,x,\eta,\eta_{\star})\in \cX\times\cV\times\sfS^+\times\sfS^+$, by:
\begin{equation}
\label{eq:def:h}
\mathsf{h}(r) = \log\left(\frac{\int \bayes(\rmd v_{p+1})\eta_{\star}(\rmd v_{p})\condlik(v_{p},v_{p+1},x)}{\int \pi(\rmd v_{p+1})\eta(\rmd v_{p})\condlik(v_{p},v_{p+1},x)}\right)\eqsp.
\end{equation}
Then, for all $i\ge 1$,
\[
\mathsf{h}(R_i) = \log\bP_{\bayes}\left(X_i\middle |X_{1:i-1}\right) - \log\bP_{\pi}\left(X_i\middle |X_{1:i-1}\right) = Z_i
\]
and for all $1\leqslant i<j\leqslant n$,
\[
\left\lvert \E_{\pi} \left[ Z_j | X_{1:i-1} \right] -\E_{\pi} \left[ Z_j \right] \right| = \left\lvert \E_{\pi} \left[ \mathsf{h}(R_j) | X_{1:i-1} \right] -\E_{\pi} \left[ \mathsf{h}(R_j) \right] \right|\eqsp.
\]
Let $\Phi$ be defined for any $\pi\in\Pi$, $x\in\cX$ and probability distribution $\eta$ on $\cV$ by:
\[
\Phi(x,\eta;\pi): A\mapsto \frac{\int\1_A(z)\pi(\rmd z)\eta(\rmd z')\condlik(z',z,x)}{\int\pi(\rmd z)\eta(\rmd z')\condlik(z',z,x)}\eqsp.
\]
For all $i\geqslant 2$, the predictive filter at time $i$ may be expressed as:
\begin{equation}
\label{eq:rec:filt}
\eta^\pi_i = \Phi(X_{i-1},\eta^\pi_{i-1};\pi)\eqsp.
\end{equation}
The transition kernel of the extended chain is given, for all $(x,v,\eta,\eta_{\star})\in \cV\times \cX \times \sfS^+\times \sfS^+$ and all $B(\cV)\times \mathcal{P}(\cX) \times B(\sfS^+)\times B(\sfS^+)$, by
\begin{multline}
\label{eq:extended:kernel}
Q_{\pi}(v,x,\eta,\eta_{\star},A_{\cV},A_{\cX},B_1,B_2) \\
= \sum_{x'\in\cX}\int \pi(\rmd v')\condlik(v,v',x')\1_{A_{\cX}}(x')\1_{A_{\cV}}(v')\1_{B_1}(\Phi(x,\eta;\pi))\1_{B_2}(\Phi(x,\eta_{\star};\bayes))\eqsp.
\end{multline}
For all $p\geqslant 1$, the $p$\,-\,th iterate of $Q_{\pi}$ is denoted by $Q_{\pi}^p$. For all $1\leqslant i <j \leqslant n$, using Lemma~\ref{lem:oubli_chaine_etendue},
\begin{align}
\left\lvert \E_{\pi} \left[ Z_j | X_{1:i-1} \right] -\E_{\pi} \left[ Z_j \right] \right\rvert
&\leqslant \int\left|\left|Q_{\pi}^{j-i}\mathsf{h}(r)-Q_{\pi}^{j-i}\mathsf{h}(\tilde{r})\right|\right|p_\pi\left(\rmd r_i\middle|X_{1:i-1}\right)p_\pi\left(\rmd \tilde{r}_i\right)\eqsp,\nonumber\\
&\leqslant
6\nu^{-1} (1-\nu)^{j-i-2}\eqsp.\label{eq:deltaZ}
\end{align}
On the other hand, by Lemma~\ref{lem:IncrementsGen}, for all $1\leqslant j \leqslant n$,
\begin{equation}
\label{eq:Z2}
\E_{\pi} \left[ Z_j^2 \right]
\leqslant
4 \nu^{-6} (2+\nu^{-1})^2\|\pi-\pi'\|_{\mathsf{tv}}^2\eqsp,
\end{equation}
so that for all $0<\beta<1$, using \eqref{eq:var}, \eqref{eq:deltaZ} and \eqref{eq:Z2},
\begin{align*}
&\text{Var}_{\pi}  \left[{\log \bP_{\pi}(X_{1:n}) - \log \bP_{\pi'}(X_{1:n})} \right] \\
 &\hspace{1.5cm} \leqslant
c_\nu\sum_{1\leqslant i\leqslant n} \E_{\pi} \left[ Z_i^2 \right] 
+ c_\nu\sum_{1\leqslant i<j\leqslant n} \sqrt{\E_{\pi} \left[ Z_j^2 \right]}
\sqrt{ \left\lvert \E_{\pi} \left[ Z_j | X_{1:i-1} \right] -\E_{\pi} \left[ Z_j \right] \right\rvert^{\beta}\E_{\pi} \left[ Z_j^2 \right]^{1-\beta/2}  }\eqsp,
\\
& \hspace{1.5cm}\leqslant 
c_\nu \beta^{-2} \|\pi-\pi'\|_{\mathsf{tv}}^{2-\beta/2} n
\eqsp,
\end{align*}
where $c_\nu$ is a constant which depends on $\nu$ only. The proof is completed by noting that the best upper bound is obtained for $\beta = -4/\log (\|\pi-\pi'\|_{\mathsf{tv}})$.

\subsection{Proof of Proposition~\ref{prop:suppi}}
\label{sec:proof:deltapi}
For all $\pi,\pi',\pi''\in\widetilde\Pi$, by \eqref{eq:def:Deltapipi'},
\begin{align*}
\Delta^n_{\pi,\pi'}(X_{1:n})-\Delta^n_{\pi,\pi''}(X_{1:n}) &= \frac{1}{n}\ell_n(\pi',X_{1:n})-\frac{1}{n}\ell_n(\pi'',X_{1:n})-\kullback_{\pi}(\pi')+\kullback_{\pi}(\pi'')\eqsp,\\
&=\;\frac{1}{n}\ell_n(\pi',X_{1:n}) - \E_{\pi}\left[\frac{1}{n}\ell_n(\pi',X_{1:n})\right]
+\E_{\pi}\left[\frac{1}{n}\ell_n(\pi',X_{1:n})\right]-\kullback_{\pi}(\pi')\\
&-\frac{1}{n}\ell_n(\pi'',X_{1:n}) + \E_{\pi}\left[\frac{1}{n}\ell_n(\pi'',X_{1:n})\right]-\E_{\pi}\left[\frac{1}{n}\ell_n(\pi'',X_{1:n})\right]+\kullback_{\pi}(\pi'')\eqsp.
\end{align*}
Note first that by Lemma~\ref{lem:likelihoodcontrol}, for all $\pi\in\widetilde\Pi$,
\[
\underset{\pi'\in\widetilde\Pi}{\sup}\left|\E_{\pi}\left[\frac{1}{n}\ell_n(\pi',X_{1:n})\right]-\kullback_{\pi}(\pi')\right|\leqslant \frac{1}{n\nu^2}\eqsp.
\]
Therefore, for all $\pi\in\widetilde\Pi$,
\begin{equation}
\label{eq:Deltapipi:G}
\underset{\pi',\pi''\in\widetilde\Pi}{\sup}\left|\Delta^n_{\pi,\pi'}(X_{1:n})-\Delta^n_{\pi,\pi''}(X_{1:n})\right|\leqslant \frac{2}{n\nu^2} + 2G^n_{\pi}(X_{1:n})\eqsp,
\end{equation}
where 
\[
G^n_{\pi}(X_{1:n}) = \underset{\pi'\in\widetilde\Pi}{\sup}\left|Z^n_{\pi,\pi'}(X_{1:n})\right|\eqsp,\quad Z^n_{\pi,\pi'}(X_{1:n}) = n^{-1}\ell_n(\pi',X_{1:n}) - n^{-1}\E_{\pi}\left[\ell_n(\pi',X_{1:n})\right] \eqsp.
\]
$G^n_{\pi}$ is a function of the Markov chain $(X_i,V_{i+1})_{1\leqslant i\leqslant n}$ whose transition kernel $P_\pi$ is uniformly lower bounded by \eqref{eq:strong:mixing}. By Corollary~\ref{cor:ConcULBMC}, it is enough to obtain a bounded difference inequality to establish a concentration inequality for $G^n_{\pi}$. For all $1\leqslant k\leqslant n$ and all $x_{1:n}\in\cX^n$ define $\tilde x^{(k)}\in\cX^n$ such that
\begin{equation}
\label{eq:def:xtilde}
\tilde x^{(k)}_k\neq x_k\quad\mbox{and\,for\,all\,} j\neq k\,,\tilde x^{(k)}_j = x_j\eqsp.
\end{equation}
Note that, for any $\pi'\in\widetilde\Pi$,
\begin{align*}
\left|Z^n_{\pi,\pi'}(x_{1:n})-Z^n_{\pi,\pi'}(\tilde x^{(k)}_{1:n})\right| &\leqslant n^{-1}\left|\ell_n(\pi',x_{1:n})-\ell_n(\pi',\tilde x^{(k)}_{1:n})\right|\eqsp,\\
&\leqslant n^{-1}\sum_{i=1}^{n}\left|\log\bP_{\pi'}(x_i|x_{1:i-1})-\log\bP_{\pi'}(\tilde x^{(k)}_i|\tilde x^{(k)}_{1:i-1})\right|\eqsp,\\
&\leqslant n^{-1}\sum_{i=k}^{n}\left|\log\bP_{\pi'}(x_i|x_{1:i-1})-\log\bP_{\pi'}(\tilde x^{(k)}_i|\tilde x^{(k)}_{1:i-1})\right|\eqsp,
\end{align*}
as for $1\leqslant i \leqslant k-1$, $x_i=\tilde x^{(k)}_i$. In the case $i=k$, 
\[
|\bP_{\pi'}(x_i|x_{1:i-1})-\bP_{\pi'}(\tilde x^{(k)}_i|\tilde x^{(k)}_{1:i-1})| = |\bP_{\pi'}(x_k|x_{1:k-1})-\bP_{\pi'}(\tilde x^{(k)}_k|x_{1:k-1})| \leqslant 1\eqsp,
\]
so that $|\log\bP_{\pi'}(x_i|x_{1:i-1})-\log\bP_{\pi'}(\tilde x^{(k)}_i|\tilde x^{(k)}_{1:i-1})|\leqslant \nu^{-1}$. For all $k+1 \leqslant i\leqslant n$, by Lemma~\ref{lem:minorization},
\[
\bP_{\pi'}(\tilde x^{(k)}_i|\tilde x^{(k)}_{1:i-1}) = \int \bP_{\pi'}(\rmd v_{k+1}|\tilde x^{(k)}_{1:i-1})\prod_{j=k+1}^{i-1}K^{V|X}_{\pi',j,i-1}(v_{j},\rmd v_{j+1})\pi'(\rmd v_{i+1})\condlik(v_i,v_{i+1},x_{i})
\]
and
\[
\bP_{\pi'}(x_i|x_{1:i-1}) = \int \bP_{\pi'}(\rmd v_{k+1}|x_{1:i-1})\prod_{j=k+1}^{i-1}K^{V|X}_{\pi',j,i-1}(v_{j},\rmd v_{j+1})\pi'(\rmd v_{i+1})\condlik(v_i,v_{i+1},x_{i})\eqsp.
\]
Therefore, $|\bP_{\pi'}(\tilde x^{(k)}_i|\tilde x^{(k)}_{1:i-1})-\bP_{\pi'}(x_i|x_{1:i-1})|\leqslant (1-\nu)^{i-k-1}$ and 
\begin{equation}
\label{eq:deltaPxxtilde}
|\log\bP_{\pi'}(\tilde x^{(k)}_i|\tilde x^{(k)}_{1:i-1}) - \log\bP_{\pi'}(x_i|x_{1:i-1})|\leqslant \nu^{-1}(1-\nu)^{i-k-1}
\end{equation}
which yields, for any $\pi'\in\widetilde\Pi$,
\[
\left|Z^n_{\pi,\pi'}(x_{1:n})-Z^n_{\pi,\pi'}(\tilde x^{(k)}_{1:n})\right|\leqslant n^{-1}\left(\nu^{-1} + \sum_{i=k+1}^{n} \nu^{-1}(1-\nu)^{i-k-1}\right)\leqslant \frac{2}{n\nu^{2}} \eqsp.
\]
By Corollary~\ref{cor:ConcULBMC} applied with $\gamma_1=\ldots=\gamma_n = 2/(n\nu^2)$, for all $\pi,\pi'\in\widetilde\Pi$ and all $t>0$,
\begin{equation}
\label{eq:conc:G}
\bP_{\pi}\left(\left|Z^n_{\pi,\pi'}(X_{1:n})\right|\geqslant \frac{2\sqrt{5}t}{\nu^3\sqrt{n}}\right)\leqslant 2\mathrm{e}^{-t^2}\eqsp.
\end{equation}
By Lemma~\ref{lem:IncrementsGen}, 
\begin{align}\label{eq:Lip}
\left|Z_{\pi,\pi_1}(x_{1:n}) -Z_{\pi,\pi_2}(x_{1:n})\right|\leqslant 4\nu^{-3}\left(2+\nu^{-1}\right)\|\pi_1-\pi_2\|_{\mathsf{tv}}\eqsp.
\end{align}
Let $\sfR(\varepsilon)$ denote an $\varepsilon$-net of $\widetilde\Pi$. By \eqref{eq:Lip},
\[
G^n_{\pi}(X_{1:n})\leqslant \underset{\pi'\in \sfR(\varepsilon)}{\max}|Z_{\pi,\pi'}(x_{1:n})\rvert +4\nu^{-3}\left(2+\nu^{-1}\right)\varepsilon\eqsp.
\]
Applying a union bound to \eqref{eq:conc:G} yields
\[
\bP_{\pi}\left(G^n_{\pi}(X_{1:n})\geqslant \frac{2\sqrt{5}t}{\nu^3\sqrt{n}}+4\nu^{-3}\left(2+\nu^{-1}\right)\varepsilon\right)\leqslant 2\sfN(\widetilde\Pi,\|.\|_{\mathrm{tv}},\varepsilon)\mathrm{e}^{-t^2}\eqsp.
\]

\subsection{Proof of Theorem~\ref{th:thetheorem}}
\label{sec:proof:th}
For $n\geqslant 1$, 
\[
\mu\left(\mathsf{B}^c_{\bar c,\star}(\varepsilon_n) \middle|X_{1:n}\right) = \frac{\int_{\mathsf{B}^c_{\bar c,\star}(\varepsilon_n)} \exp\left\{\ell_n(\pi,X_{1:n})-\ell_n(\bayes,X_{1:n})\right\} \mu(\rmd \pi)}{\int_{\Pi} \exp\left\{\ell_n(\pi,X_{1:n})-\ell_n(\bayes,X_{1:n})\right\} \mu(\rmd \pi)} = \frac{N_n}{D_n}\eqsp.
\]
Following \cite[Section~3.3]{rousseau:2016}, consider, for all $u>0$ and all function $g:\mathbb{R}_+^{\star}\to \mathbb{R}_+^{\star}$, the decomposition
\begin{multline*}
\bP_{\bayes}\left(\mu\left(\mathsf{B}^c_{\bar c,\star}(\varepsilon_n) \middle|X_{1:n}\right)>\alpha_n\right)
\leqslant 
\E_{\bayes}\left[\Phi^{\tau_{n,c}(u)}_n(X_{1:n},\bayes)\right] + \bP_{\bayes}\left(D_n<g(u)\right)\\
 + \bP_{\bayes}\left(\left\{\Phi^{\tau_{n,c}(u)}_n(X_{1:n},\bayes)=0\right\}\cap\left\{D_n\ge g(u)\right\}\cap\left\{\frac{N_n}{D_n}>\alpha_n\right\}\right)\eqsp,
\end{multline*}
where for all $n\geqslant 1$ and all $t>0$,
\begin{equation}
\label{eq:def:test}
\Phi^t_n(X_{1:n},\bayes) = \mathds{1}_{\{\exists \pi\in \Pi_n\;:\;\ell_n(\pi,X_{1:n})/n-\ell_n(\bayes,X_{1:n})/n + \kullback_{\bayes}(\bayes)-\kullback_{\bayes}(\pi)>t\}}
\end{equation}
and for all $u>0$, $c>0$ and $n\geqslant 1$
\[
\tau_{n,c}(u) = c \left(n^{-1/2}u+\varepsilon_n\right)\eqsp.
\]
By Markov inequality, for all $u>0$,
\begin{align*}
\bP_{\bayes}\left(\left\{\Phi^{\tau_{n,c}(u)}_n(X_{1:n},\bayes)=0\right\}\cap\left\{D_n\ge g(u)\right\}\cap\left\{\frac{N_n}{D_n}>\alpha_n\right\}\right)&\\
&\hspace{-9cm}\leqslant \bP_{\bayes}\left(\left\{\Phi^{\tau_{n,c}(u)}_n(X_{1:n},\bayes)=0\right\}\cap\left\{N_n>\alpha_n g(u)\right\}\right)\eqsp,\\
&\hspace{-9cm}\leqslant \bP_{\bayes}\left(\left\{\int_{\mathsf{B}^c_{\bar c,\star}(\varepsilon_n)} \left(1-\Phi^{\tau_{n,c}(u)}_n(X_{1:n},\bayes)\right)\frac{\exp\left\{\ell_n(\pi,X_{1:n})\right\}}{\exp\left\{\ell_n(\bayes,X_{1:n})\right\}}\mu(\rmd \pi)>\alpha_n g(u)\right\}\right)\eqsp,\\
&\hspace{-9cm}\leqslant \frac{1}{\alpha_n g(u)}\E_{\bayes}\left[\int_{\mathsf{B}^c_{\bar c,\star}(\varepsilon_n)} \left(1-\Phi^{\tau_{n,c}(u)}_n(X_{1:n},\bayes)\right)\frac{\exp\left\{\ell_n(\pi,X_{1:n})\right\}}{\exp\left\{\ell_n(\bayes,X_{1:n})\right\}}\mu(\rmd \pi)\right]\eqsp,\\
&\hspace{-9cm}\leqslant \frac{1}{\alpha_n g(u)}\int_{\mathsf{B}^c_{\bar c,\star}(\varepsilon_n)} \E_{\pi}^n\left[1-\Phi^{\tau_{n,c}(u)}_n(X_{1:n},\bayes)\right]\mu(\rmd \pi)\eqsp.
\end{align*}
Therefore,
\begin{multline*}
\bP_{\bayes}\left(\mu\left(\mathsf{B}^c_{\bar c,\star}(\varepsilon_n) \middle|X_{1:n}\right)>\alpha_n\right)
\leqslant 
\E_{\bayes}\left[\Phi^{\tau_{n,c}(u)}_n(X_{1:n},\bayes)\right] + \bP_{\bayes}\left(D_n< g(u)\right)\\
 + \frac{1}{\alpha_n g(u)}\int_{\mathsf{B}^c_{\bar c,\star}(\varepsilon_n)\cap \Pi_n} \E_{\pi}\left[1-\Phi^{\tau_{n,c}(u)}_n(X_{1:n},\bayes)\right]\mu(\rmd \pi) + \frac{\mu\left(\Pi_n^c\right)}{\alpha_n g(u)}\eqsp.
\end{multline*}
Then,
\begin{align*}
\E_{\bayes}\left[\Phi^{\tau_{n,c}(u)}_n(X_{1:n},\bayes)\right] &\\
 & \hspace{-2cm} =  \bP_{\bayes}\left(\exists \pi\in \Pi_n\;:\;(\ell_n(\pi,X_{1:n})-\ell_n(\bayes,X_{1:n}))/n-(\kullback_{\bayes}(\pi)-\kullback_{\bayes}(\bayes))>\tau_{n,c}(u)\right)\eqsp,\\
 & \hspace{-2cm}\leqslant \bP_{\bayes}\left(\exists \pi\in \Pi_n\;:\;\left|\Delta^n_{\bayes,\pi}(X_{1:n})-\Delta^n_{\bayes,\bayes}(X_{1:n})\right|>\tau_{n,c}(u)\right)\eqsp,
\end{align*}
where $\Delta^n_{\bayes,\pi}(X_{1:n})$ and $\Delta^n_{\bayes,\bayes}(X_{1:n})$ are defined by \eqref{eq:def:Deltapipi'}. 
Then, by Proposition~\ref{prop:suppi} with $\widetilde \Pi = \Pi_n\cup\{\bayes\}$, choosing $c=(1+\sqrt{2})c_{\nu}$ and $u=u_n=(\bar c/(2c) - 1)\sqrt{n}\tilde{\varepsilon}_n$, for any sufficiently large $n$,
\[
\E_{\bayes}\left[\Phi^{\tau_{n,c}(u_n)}_n(X_{1:n},\bayes)\right]\leqslant \mathrm{e}^{-u_n^2}\eqsp.
\]
This yields
\begin{multline*}
\bP_{\bayes}\left(\mu\left(\mathsf{B}^c_{\bar c,\star}(\varepsilon_n) \middle|X_{1:n}\right)>\alpha_n\right)
\leqslant 
\mathrm{e}^{-u_n^2} + \bP_{\bayes}\left(D_n<g(u_n)\right)\\
 + \frac{1}{\alpha_n g(u_n)}\int_{\mathsf{B}^c_{\bar c,\star}(\varepsilon_n)\cap \Pi_n} \E_{\pi}\left[1-\Phi^{\tau_{n,c}(u_n)}_n(X_{1:n},\bayes)\right]\mu(\rmd \pi) + \frac{\mu\left(\Pi_n^c\right)}{\alpha_n g(u_n)}\eqsp.
\end{multline*}
Then, using that for all $\pi\in\Pi_n$, $\kullback_{\pi}(\pi)-\kullback_{\pi}(\bayes)\geqslant 0$, and for all $\pi\in\mathsf{B}^c_{\bar c,\star}(\varepsilon_n)$, $\kullback_{\bayes}(\bayes)-\kullback_{\bayes}(\pi)>\bar c\varepsilon_n$, 
\begin{align*}
\E_{\pi}\left[1-\Phi^{\tau_{n,c}(u_n)}_n(X_{1:n},\bayes)\right] & \\
&\hspace{-3cm}= \bP_{\pi}\left(\forall \pi'\in \Pi_n\;:\;\ell_n(\pi',X_{1:n})/n-\ell_n(\bayes,X_{1:n})/n-(\kullback_{\bayes}(\pi')-\kullback_{\bayes}(\bayes))\leqslant \tau_{n,c}(u_n)\right)\eqsp,\\
&\hspace{-3cm} \leqslant \bP_{\pi}\left(\ell_n(\pi,X_{1:n})/n-\ell_n(\bayes,X_{1:n})/n\leqslant \tau_{n,c}(u_n)-\bar c\varepsilon_n\right)\eqsp,\\
 &\hspace{-3cm} \leqslant \bP_{\pi}\left((\ell_n(\pi,X_{1:n})-\ell_n(\bayes,X_{1:n}))/n-(\kullback_{\pi}(\pi)-\kullback_{\pi}(\bayes))\leqslant \tau_{n,c}(u_n)-\bar c\varepsilon_n\right)\eqsp.
\end{align*}
Since $\tau_{n,c}(u_n) \leq \bar c\varepsilon_n/2$, 
\begin{align*}
\E_{\pi}\left[1-\Phi^{\tau_{n,c}(u_n)}_n(X_{1:n},\bayes)\right] & \leqslant\bP_{\pi}\left(\left|\Delta^n_{\pi,\pi}(X_{1:n})-\Delta^n_{\pi,\bayes}(X_{1:n})\right|\geqslant\bar c\varepsilon_n/2\right)\eqsp,\\
& \leqslant\bP_{\pi}\left(\left|\Delta^n_{\pi,\pi}(X_{1:n})-\Delta^n_{\pi,\bayes}(X_{1:n})\right|\geqslant\tau_{n,c}(u_n)\right)\eqsp,\\
&\leqslant \mathrm{e}^{-u_n^2}\eqsp,
\end{align*}
where the last inequality follows from Proposition~\ref{prop:suppi}. 
Then,
\[
\bP_{\bayes}\left(\mu\left(\mathsf{B}^c_{\bar c,\star}(\varepsilon_n) \middle|X_{1:n}\right)>\alpha_n\right)
\leqslant 
\mathrm{e}^{-u_n^2} + \bP_{\bayes}\left(D_n<g(u_n)\right)
 + \frac{\mathrm{e}^{-u_n^2}}{\alpha_n g(u_n)} + \frac{\mu\left(\Pi_n^c\right)}{\alpha_n g(u_n)}\eqsp.
\]
Let $c_3>c_\nu$ where $c_\nu$ is given in Proposition~\ref{prop:controle_KL_et_var_ln}. 
Since
\[
D_n \geqslant \int_{\mathsf{S}_{\star}(\tilde{\varepsilon}_n)} \exp\left\{\ell_n(\pi,X_{1:n})-\ell_n(\bayes,X_{1:n})\right\}\1_{\{\ell_n(\pi,X_{1:n})-\ell_n(\bayes,X_{1:n})\geqslant -c_3n\tilde{\varepsilon}_n^2\}} \mu(\rmd \pi)\eqsp,
\]
then 
\begin{align*}
\bP_{\bayes}\left(D_n<g(u_n)\right) &\\
&\hspace{-1.7cm}\leqslant \bP_{\bayes}\left(\mu\left(\mathsf{S}_{\star}(\tilde{\varepsilon}_n)\cap \left\{\pi : \ell_n(\pi,X_{1:n})-\ell_n(\bayes,X_{1:n})\geqslant -c_3n\tilde{\varepsilon}_n^2\right\}\right)<g(u_n)\mathrm{e}^{c_3n\tilde{\varepsilon}_n^2}\right)\eqsp,\\
 &\hspace{-1.7cm}\leqslant \bP_{\bayes}\left(\mu\left(\mathsf{S}_{\star}(\tilde{\varepsilon}_n)\cap \left\{\pi : \ell_n(\pi,X_{1:n})-\ell_n(\bayes,X_{1:n})\geqslant -c_3n\tilde{\varepsilon}_n^2\right\}^c\right)>\mu\left(\mathsf{S}_{\star}(\tilde{\varepsilon}_n)\right)-g(u_n)\mathrm{e}^{c_3n\tilde{\varepsilon}_n^2}\right)\eqsp.
\end{align*}
By Markov inequality,
\begin{align*}
\bP_{\bayes}\left(D_n<g(u_n)\right)&\leqslant \frac{\E_{\bayes}\left[\mu\left(\mathsf{S}_{\star}(\tilde{\varepsilon}_n)\cap \left\{\pi : \ell_n(\pi,X_{1:n})-\ell_n(\bayes,X_{1:n})\geqslant -c_3n\tilde{\varepsilon}_n^2\right\}^c\right)\right]}{\mu\left(\mathsf{S}_{\star}(\tilde{\varepsilon}_n)\right)-g(u_n)\mathrm{e}^{c_3n\tilde{\varepsilon}_n^2}}\eqsp,\\
&\leqslant \frac{\int_{\mathsf{S}_{\star}(\tilde{\varepsilon}_n)}\bP_{\bayes}\left(\ell_n(\pi,X_{1:n})-\ell_n(\bayes,X_{1:n})< -c_3n\tilde{\varepsilon}_n^2\right)\mu(\rmd \pi)}{\mu\left(\mathsf{S}_{\star}(\tilde{\varepsilon}_n)\right)-g(u_n)\mathrm{e}^{c_3n\tilde{\varepsilon}_n^2}}\eqsp,\\
&\leqslant \frac{\mu\left(\mathsf{S}_{\star}(\tilde{\varepsilon}_n)\right)}{\mu\left(\mathsf{S}_{\star}(\tilde{\varepsilon}_n)\right)-g(u_n)\mathrm{e}^{c_3n\tilde{\varepsilon}_n^2}}\frac{c_\nu \log^2(\tilde{\varepsilon}_n)}{(c_3-c_\nu)^2n\tilde{\varepsilon}_n^2}\eqsp,
\end{align*}
where the last inequality follows from Proposition~\ref{prop:controle_KL_et_var_ln} and \cite[Lemma 10]{ghosal:vandervaart:2007}. 
Choosing $g:u\mapsto \mathrm{e}^{-\tau u^2}/2$ with $\tau = (c_3+c_1)/(\bar c/(2c)-1)^2 $ yields $g(u_n)\mathrm{e}^{c_3n\tilde{\varepsilon}_n^2} = \mathrm{e}^{-c_1n\tilde{\varepsilon}_n^2}/2$ and by H\ref{assum:entropy}-$(c_0,c_1,c_2)$,
\[
\bP_{\bayes}\left(D_n<g(u_n)\right)\leqslant \frac{2c_\nu \log^2(\tilde\varepsilon_n)}{(c_3-c_\nu)^2n\tilde\varepsilon_n^2}\eqsp.
\]
Then, with $c_3=c_\nu+ \left( c_2 - c_{\nu} - c_1 -c_\alpha \right)/2>c_\nu$,
\[
\frac{\mu\left(\Pi_n^c\right)}{\alpha_n g(u_n)} \leqslant \frac{2}{\alpha_n}\mathrm{e}^{\left[(c_3+c_1-c_2)\right]n\tilde{\varepsilon}_n^2}=o(1)\eqsp,
\]
which concludes the proof.

\section{Forgetting properties of the loglikelihood}

\begin{lemma}
\label{lem:minorization}
Assume that  H\ref{assum:lowerbound} holds. For any $i\geqslant 1$, conditionally on $X_{1:i}$, $(V_k)_{k\geqslant 1}$ is a Markov chain. 
Its transition kernels $(K^{V|X}_{\pi,k,i})_{k\geqslant 1}$ are such that, for all $1\leqslant k \leqslant i$, there exists a measure $\mu_{k,i}$ satisfying for all measurable set $A$:
\begin{align*}
K^{V|X}_{\pi,k,i}(V_{k},A) = \bP_{\pi}\pa{V_{k+1}\in A\middle|V_{1:k},X_{1:i}}  &= \bP_{\pi}\pa{V_{k+1}\in A\middle|V_{k},X_{k:i}}\geqslant \nu \mu_{k,i}(A)\eqsp,
\end{align*}
where $\mu_{i,i} = \pi$ and for $1\leqslant k <i$,
\[
\mu_{k,i}(A) = \frac{\int \1_{A}(v_{k+1})\pi(\rmd v_{k+1})\bP_{\pi}(X_{k+1:i}|v_{k+1})}{\int \pi(\rmd v_{k+1})\bP_{\pi}(X_{k+1:i}|v_{k+1})}\eqsp.
\]
On the other hand, for all $k \geqslant i+1$ and all measurable set $A$,
\[
K^{V|X}_{\pi,k,i}(V_{k},A) = \bP_{\pi}\pa{V_{k+1}\in A\middle|V_{1:k},X_{1:i}} = \pi(A)\eqsp.
\]
\end{lemma}
\begin{proof}
The proof is a classical result for strong mixing latent data models, it follows closely \cite{douc:moulines:ryden:2004}. 
See also \cite{diel:lecorff:lerasle:2018} for a proof in the framework described by Figure~\ref{fig:generic:graphicalmodel}.
\end{proof}

\begin{lemma}
\label{lem:likelihoodcontrol}
Assume that  H\ref{assum:lowerbound} holds. For all distribution $\pi$ and all $i\geqslant 1$, $\ell\geqslant 0$,
\[
\left|\log\bP_{\pi}\left(x_{i}\middle|x_{1:i-1}\right)-\log\bP_{\pi}\left(x_{i}\middle|x_{-\ell:i-1}\right)\right|\leqslant \nu^{-1}(1-\nu)^{i-1}\eqsp.
\]
There exists a function $\ell_{\pi}$ such that, for all distribution $\pi'$, $\bP_{\pi'}$-a.s.  and in $\mathrm{L}^1(\bP_{\pi'})$,
\[
\frac{1}{n}\ell_n(\pi,X_{1:n}) \limit{n} \kullback_{\pi'}(\pi) = \E_{\pi'}\left[\ell_{\pi}(X)\right]
\]
and
\[
\left|\E_{\pi'}\left[\frac{1}{n}\ell_n(\pi,X_{1:n})\right]-\kullback_{\pi'}(\pi)\right| \leqslant \frac{1}{n\nu^2}\eqsp.
\]
\end{lemma}

\begin{proof}
Let $x\in\cX^{\mathbb{Z}}$. For all $i\geqslant 2$ and all $\ell \geqslant 0$,
\begin{align*}
\bP_{\pi}\left(x_{i}\middle|x_{1:i-1}\right) &= \int \bP_{\pi}\left(\rmd v_{1}\middle | x_{1:i-1}\right) \pa{\prod_{k = 1}^{i-1}K^{V|X}_{\pi,k,i-1}(v_{k},\rmd v_{k+1})}\pi(\rmd v_{i+1})\condlik(x_i,v_i,v_{i+1})\eqsp,\\
\bP_{\pi}\left(x_{i}\middle|x_{-\ell:i-1}\right) &= \int \bP_{\pi}\left(\rmd v_{1}\middle | x_{-\ell:i-1}\right) \pa{\prod_{k = 1}^{i-1}K^{V|X}_{\pi,k,i-1}(v_{k},\rmd v_{k+1})}\pi(\rmd v_{i+1})\condlik(x_i,v_i,v_{i+1})\eqsp,
\end{align*}
as $V_{2:i+1}$ is independent of $X_{-\ell:0}$ conditionally on $V_1$. Therefore, by Lemma~\ref{lem:minorization},
\[
\left|\bP_{\pi}\left(x_{i}\middle|x_{1:i-1}\right)-\bP_{\pi}\left(x_{i}\middle|x_{-\ell:i-1}\right)\right|\leqslant (1-\nu)^{i-1}
\]
and, since $\bP_{\pi}\left(x_{i}\middle|x_{1:i-1}\right) \wedge \bP_{\pi}\left(x_{i}\middle|x_{-\ell:i-1}\right) \geqslant \nu$, 
\[
\left|\log \bP_{\pi}\left(x_{i}\middle|x_{1:i-1}\right)- \log\bP_{\pi}\left(x_{i}\middle|x_{-\ell:i-1}\right)\right|\leqslant \nu^{-1}(1-\nu)^{i-1}\eqsp.
\]
Similarly, for $i=1$,
\[
\left|\log \bP_{\pi}\left(x_{1}\right)- \log\bP_{\pi}\left(x_{1}\middle|x_{-\ell:0}\right)\right|\leqslant \nu^{-1}\eqsp.
\]
Let $(X_i,V_{i+1})_{i\in\bZ}$ be a process with distribution $\bP_{\pi'}$. For all $i\in\mathbb{Z}$,
\[
\left|\log \bP_{\pi}\left(X_{i}\middle|X_{1:i-1}\right)- \log\bP_{\pi}\left(X_{i}\middle|X_{-\ell:i-1}\right)\right|\leqslant \nu^{-1}(1-\nu)^{i-1}\eqsp.
\]
Define the shift operator $\shift$ on $\cX^{\bZ}$ by $(\shift x)_i = x_{i+1}$ for all $i\in\bZ$ and all $x\in\cX^{\bZ}$. There exists a function $\ell_{\pi}$ such that for all $i\in\bZ$ the sequence $(\log\bP_{\pi}(X_{i}|X_{-\ell:i-1}))_{\ell\geqslant 0}$ converges $\bP_{\pi'}$-a.s. to $\ell_{\pi}\left(\shift^i X\right)$. Then, when $\ell$ grows to $\infty$, this yields
\[
\left|\log \bP_{\pi}\left(X_{i}\middle|X_{1:i-1}\right)- \ell_{\pi}\left(\shift^i X\right)\right|\leqslant \nu^{-1}(1-\nu)^{i-1}\eqsp.
\]
Then, 
\[
\left|\frac{1}{n}\ell_n(\pi,X_{1:n}) - \frac{1}{n}\sum_{i=1}^n\ell_{\pi}\left(\shift^i X\right)\right| = \left|\frac{1}{n}\sum_{i=1}^n\{\log \bP_{\pi}\left(X_{i}\middle|X_{1:i-1}\right)- \ell_{\pi}\left(\shift^i X\right)\}\right| \leqslant \frac{1}{n\nu^2}
\]
and the proof of the second claim is completed with the ergodic theorem \cite[Theorem~24.1]{billingsley:1995}.
Remarking that for all $i\geqslant 1$, $\E_{\pi'}\left[\ell_{\pi}\left(\shift^i X\right)\right] = \kullback_{\pi'}(\pi)$, the loglikelihood of $X_{1:n}$ also satisfies
\begin{equation*}
\left|\E_{\pi'}\left[\frac{1}{n}\ell_n(\pi,X_{1:n})\right]-\kullback_{\pi'}(\pi)\right| = \left|\E_{\pi'}\left[\frac{1}{n}\sum_{i=1}^n \left\{\log \bP_{\pi}\left(x_{i}\middle|x_{1:i-1}\right) - \ell_{\pi}\left(\shift^i X\right)\right\}\right]\right|\leqslant \frac{1}{n\nu^2}\eqsp.
\end{equation*}
\end{proof}

\begin{lemma}\label{lem:IncrementsGen}
Assume that  H\ref{assum:lowerbound} holds. 
For all $x_{1:n}\in\cX^n$, all distributions $\pi,\pi'\in\Pi$ and any $1\leqslant i\leqslant n$,
\[
\absj{ \bP_{\pi}(x_{i}|x_{1:i-1}) -  \bP_{\pi'}(x_{i}|x_{1:i-1})} \\
\leqslant 2\nu^{-2}(2+\nu^{-1})\|\pi-\pi'\|_{\mathsf{tv}}\eqsp.
\]
\end{lemma}
\begin{proof}
When $i=1$,
\[
\bP_{\pi}(x_{1}) - \bP_{\pi'}(x_{1}) = \int\left\{\pi\otimes\pi(\rmd v_{1:2})-\pi'{\otimes}\pi'(\rmd v_{1:2})\right\}\condlik(x_{1},v_{1},v_{2})\eqsp.
\]
Thus $|\bP_{\pi}(x_{1}) - \bP_{\pi'}(x_{1})|\le 2 \|\pi-\pi'\|_{\mathsf{tv}}$.
For all $2\leqslant i\leqslant n$,
\begin{equation*}
\bP_{\pi}(x_{i}|x_{1:i-1}) - \bP_{\pi'}(x_{i}|x_{1:i-1}) 
= \sum_{\ell= 1}^{i+1}\left\{\bP_{\ell}(x_{i}|x_{1:i-1}) - \bP_{\ell+1}(x_{i}|x_{1:i-1})\right\}\eqsp,
\end{equation*}
where $\bP_{\ell}$ is the joint distribution of $(X_{1:i},V_{1:i+1})$ when $(V_1,\ldots,V_{\ell-1})$ are i.i.d. with distribution $\pi'$ and $(V_{\ell},\ldots,V_{i+1})$ are i.i.d. with distribution $\pi$. 
Then, for all $1\leqslant \ell \leqslant i-1$,
\begin{align*}
\bP_{\ell}\left(x_{i}\middle|x_{1:i-1}\right) &= \int \bP_{\ell}\left(\rmd v_{\ell+1}\middle | x_{1:i-1}\right) \pa{\prod_{k = \ell+1}^{i}K^{V|X}_{\pi,k,i-1}(v_{k},\rmd v_{k+1})}\condlik(x_i,v_i,v_{i+1})\eqsp,\\
&= \int \bP_{\ell}\left(\rmd v_{\ell+1}\middle | x_{1:i-1}\right) \pa{\prod_{k = \ell+1}^{i-1}K^{V|X}_{\pi,k,i-1}(v_{k},\rmd v_{k+1})}\pi(\rmd v_{i+1})\condlik(x_i,v_i,v_{i+1})\eqsp.
\end{align*}
Therefore, for all $1\leqslant \ell \leqslant i-1$, by Lemma~\ref{lem:minorization}, 
\[
\left|\bP_{\ell}\left(x_{i}\middle|x_{i+1:n}\right)-\bP_{\ell+1}\left(x_{i}\middle|x_{i+1:n}\right)\right|\le \left(1-\nu\right)^{i-\ell-1}\norm{\bP_{\ell}\left(\cdot\middle | x_{1:i-1}\right)-\bP_{\ell+1}\left(\cdot\middle | x_{1:i-1}\right)}_{\mathsf{tv}}\eqsp,
\]
where $\bP_{\ell}\left(\cdot\middle | x_{1:i-1}\right)$ is the distribution of $V_{\ell+1}$ conditionally on $\{X_{1:i-1}=x_{1:i-1}\}$ when $(V_1,\ldots,V_{\ell-1})$ are i.i.d. with distribution $\pi'$ and $(V_{\ell},\ldots,V_{i+1})$ are i.i.d. with distribution $\pi$. 
We first show that 
\[
\norm{\bP_{\ell}\left(\cdot\middle | x_{1:i-1}\right)-\bP_{\ell+1}\left(\cdot\middle | x_{1:i-1}\right)}_{\mathsf{tv}}\le  2\nu^{-2} \|\pi-\pi'\|_{\mathsf{tv}}\eqsp.
\]
Write, for all $1\leqslant \ell \leqslant i-1$,
\begin{equation}
\label{eq:defL}
\mathrm{L}^{x_{1:i-1}}_{\ell}(\rmd v_{1:i})= \prod_{m=1}^{\ell-1}\pi'(\rmd v_m)\prod_{m=\ell}^{i}\pi(\rmd v_m)\prod_{m=1}^{i-1}\condlik(x_{m},v_{m},v_{m+1})\eqsp.
\end{equation}
Then, for any function $f$ on $\cV$,
\[
\int f(v_{\ell+1})\bP_{\ell}\left(\rmd v_{\ell+1}\middle | x_{1:i-1}\right)=\frac{\int f(v_{\ell+1})L^{x_{1:i-1}}_\ell(\rmd v_{1:i})}{\int L^{x_{1:i-1}}_\ell(\rmd v_{1:i})}\eqsp.
\]
Therefore,
\begin{align*}
 \int f(v_{\ell+1})\left\{\bP_{\ell}\left(\rmd v_{\ell+1}\middle | x_{1:i-1}\right)-\bP_{\ell+1}\left(\rmd v_{\ell+1}\middle | x_{1:i-1}\right)\right\} &\\
 &\hspace{-6cm}=\int f(v_{\ell+1})\pa{\frac{L^{x_{1:i-1}}_\ell(\rmd v_{1:i})}{\int L^{x_{1:i-1}}_\ell(\rmd v_{1:i})}-\frac{L^{x_{1:i-1}}_{\ell+1}(\rmd v_{1:i})}{\int L^{x_{1:i-1}}_{\ell+1}(\rmd v_{1:i})}}\eqsp,\\
 &\hspace{-6cm}= \int f(v_{\ell+1})\frac{L^{x_{1:i-1}}_\ell(\rmd v_{1:i})-L^{x_{1:q-1}}_{\ell+1}(\rmd v_{1:i})}{\int L^{x_{1:i-1}}_\ell(\rmd v_{1:i})}\\
 &\hspace{-2cm}+\int f(v_{\ell+1})\frac{L^{x_{1:i-1}}_{\ell+1}(\rmd v_{1:i})}{\int L^{x_{1:i-1}}_{\ell+1}(\rmd v_{1:i})}\frac{\int \cro{L^{x_{1:i-1}}_{\ell+1}(\rmd v_{1:i})-L^{x_{1:i-1}}_\ell( \rmd v_{1:i})}}{\int L^{x_{1:i-1}}_\ell(\rmd v_{1:i})}\eqsp.
\end{align*}
Thus, for any function $f$ on $\cV$ such that $\|f\|_{\infty}<1$,
\begin{multline}\label{eq:TV1Gen}
\left|\int f(v_{\ell+1})\left\{\bP_{\ell}\left(\rmd v_{\ell+1}\middle | x_{1:i-1}\right)-\bP_{\ell+1}\left(\rmd v_{\ell+1}\middle | x_{1:i-1}\right)\right\}\right|
\le 2\frac{|\int \{L^{x_{1:i-1}}_\ell(\rmd v_{1:i})-L^{x_{1:i-1}}_{\ell+1}(\rmd v_{1:i})\}|}{\int L^{x_{1:i-1}}_\ell(\rmd v_{1:i})}\eqsp. 
\end{multline}
By \eqref{eq:defL}, $1\leqslant \ell \leqslant i-1$,
\begin{multline*}
\left|\int\{L^{x_{1:i-1}}_\ell(\rmd v_{1:i})-L^{x_{1:i-1}}_{\ell+1}(\rmd v_{1:i})\}\right| \\
= \left|\int\prod_{m=1}^{\ell-1}\pi'(\rmd v_m)\left\{\pi(\rmd v_{\ell})-\pi'(\rmd v_{\ell})\right\}\prod_{m=\ell+1}^{i}\pi(\rmd v_m)\prod_{m=1}^{i-1}\condlik(x_{m},v_{m},v_{m+1})\right|\eqsp.
\end{multline*}
As $\condlik$ is upper bounded by 1, 
\begin{multline*}
\left|\int\{L^{x_{1:i-1}}_\ell(\rmd v_{1:i})-L^{x_{1:i-1}}_{\ell+1}(\rmd v_{1:i})\}\right| \le \left(\int\prod_{m=1}^{\ell-1}\pi'(\rmd v_m)\prod_{m=1}^{\ell-2}\condlik(x_{m},v_{m},v_{m+1})\right)\\
\times \norm{\pi-\pi'}_{\textrm{tv}}\left(\int\prod_{m=\ell+1}^{i}\pi(\rmd v_m)\prod_{m=\ell+1}^{i-1}\condlik(x_{m},v_{m},v_{m+1})\right)
\eqsp.
\end{multline*}
Similarly, since $\condlik$ is respectively lower bounded by $\nu$,
\begin{multline*}
\int L^{x_{1:i-1}}_\ell(\rmd v_{1:i}) \ge \left(\int\prod_{m=1}^{\ell-1}\pi'(\rmd v_m)\prod_{m=1}^{\ell-2}\condlik(x_{m},v_{m},v_{m+1})\right)\\
\times \nu^2\left(\int\prod_{m=\ell+1}^{i}\pi(\rmd v_m)\prod_{m=\ell+1}^{i-1}\condlik(x_{m},v_{m},v_{m+1})\right) \eqsp.
\end{multline*}
Plugging these bounds in \eqref{eq:TV1Gen} yields, for  $1\leqslant \ell \leqslant i-1$,
\begin{equation*}
\left|\int f(v_{\ell+1})\left\{\bP_{\ell}\left(\rmd v_{\ell+1}\middle | x_{1:i-1}\right)-\bP_{\ell+1}\left(\rmd v_{\ell+1}\middle | x_{1:i-1}\right)\right\}\right|\le 2\nu^{-2}\norm{\pi-\pi'}_{\textrm{tv}}\eqsp. 
\end{equation*}
Then, by bounding similarly the two last terms of the telescoping sum,
\begin{align*}
\bP_{\pi}(x_{i}|x_{1:i-1}) - \bP_{\pi'}(x_{i}|x_{1:i-1}) = 2\nu^{-2} \|\pi-\pi'\|_{\mathsf{tv}}\left\{\sum_{\ell= 1}^{i-1} \left(1-\nu\right)^{i-\ell-1} + 2\right\}\le 2(2+\nu^{-1})\nu^{-2} \|\pi-\pi'\|_{\mathsf{tv}}\eqsp.
\end{align*}
\end{proof}

\section{Forgetting properties of the predictive filter}
For all $k\geqslant 1$, $\eta^{\pi}_{k}$ is the predictive filter at time $k$ when the law of the hidden states is $\pi$: $\eta^{\pi}_{1} = \pi$ and for $k\ge 2$,
\[
\eta^{\pi}_{k}:A\mapsto \bP_{\pi}\left(V_k\in A\middle |X_{1:k-1}\right)\eqsp.
\]
By applying \eqref{eq:rec:filt} recursively, for all $p\geqslant 1$, define the function $\Phi_p$ by:
\[
\eta^\pi_{p+1} = \Phi(X_{p},\eta^\pi_{p};\pi) = \Phi(X_{p},\Phi(X_{p-1},\eta^\pi_{p-1};\pi);\pi) = \ldots = \Phi_{p}(X_{1:p},\eta_1^{\pi};\pi)\eqsp,
\]
with the convention $\Phi_1 = \Phi$. 
Lemma~\ref{lem:predictive:forgetting} establishes the exponential forgetting of the prediction filter. 
Its proof follows closely \cite[Proposition~1]{douc:matias:2001} in the case of general hidden Markov models.
\begin{lemma}
\label{lem:predictive:forgetting}
Assume that  H\ref{assum:lowerbound} holds. For all $p\geqslant 1$, all $x_{1:p}\in\cX^p$
 and all probability distributions $\eta,\eta',\pi$ on $\cV$,
\[
\|\Phi_{p}(x_{1:p},\eta;\pi) - \Phi_{p}(x_{1:p},\eta';\pi)\|_{\mathsf{tv}} \leqslant (1-\nu)^{p}\eqsp.
\]
\end{lemma}

\begin{proof}
By definition, $\Phi_{p}(x_{1:p},\eta;\pi)$ is the predictive distribution of $V_{p+1}$ given $\{X_{1:p} = x_{1:p}\}$ when $V_1\sim \eta$ and $V_k\sim\pi$ for $2\le k \le p+1$. 
By Lemma~\ref{lem:minorization},
\[
\int \Phi_{p}(x_{1:p},\eta;\pi)(\rmd v_{p+1})h(v_{p+1}) = \int \phi_{\eta,\pi,p}(\rmd v_1)\prod_{k=1}^{p}K^{V|X}_{\pi,k,p}(v_{k},\rmd v_{k+1})h(v_{p+1})\eqsp,
\]
where $\phi_{\eta,\pi,p}$ is the conditional distribution of $V_1$ given $\{X_{1:p} = x_{1:p}\}$ when $V_1\sim \eta$ and $V_k\sim\pi$ for $2\le k \le p+1$. 
Therefore, by Lemma~\ref{lem:minorization},
\[
\|\Phi_{p}(x_{1:p},\eta;\pi) - \Phi_{p}(x_{1:p},\eta';\pi)\|_{\mathsf{tv}} \leqslant (1-\nu)^{p}\|\phi_{\eta,\pi,p}-\phi_{\eta',\pi,p}\|_{\mathsf{tv}}\eqsp,
\]
which concludes the proof.
\end{proof}

\begin{lemma}\label{lem:oubli_chaine_etendue}
Assume H\ref{assum:lowerbound} holds. Then, for all $r=(v,x,\eta,\eta_{\star})\in \cV\times\cX\times\sfS^+\times\sfS^+$, $r'=(v',x',\eta',\eta_{\star}')\in \cV\times\cX\times\sfS^+\times\sfS^+$ and all $p\geqslant 1$,
\[
\left|Q_{\pi}^p\mathsf{h}(r)-Q_{\pi}^p\mathsf{h}(r')\right| \leqslant 6\nu^{-1}(1-\nu)^{p-2}\eqsp,
\]
where $\mathsf{h}$ and $Q_{\pi}$ are defined in \eqref{eq:def:h} and \eqref{eq:extended:kernel}.
\end{lemma}

\begin{proof}
For all measurable function $h$, $r=(v,x,\eta,\eta_{\star})\in \cX\times\cV\times\sfS^+\times\sfS^+$, $r'=(v',x',\eta',\eta_{\star}')\in \cX\times\cV\times\sfS^+\times\sfS^+$ and all $p\ge 1$,
\begin{multline}
\label{eq:decomp:Qph}
\left|Q_{\pi}^ph(r)-Q_{\pi}^ph(r')\right| \le \left|Q_{\pi}^ph(v,x,\eta,\eta_{\star})-Q_{\pi}^ph(v,x,\eta',\eta_{\star}')\right|\\
+ \left|Q_{\pi}^ph(v,x,\eta',\eta_{\star}')-Q_{\pi}^ph(v',x',\eta',\eta_{\star}')\right|\eqsp.
\end{multline}
With the convention $v_2 = v$, for all $p\geqslant 1$,
\begin{multline*}
Q_{\pi}^ph(r) = \sum_{x_{2:p+1}\in\cX^{p}}\int h(v_{p+2},x_{p+1},\Phi_{p}((x,x_{2:p}),\eta;\pi),\Phi_{p}((x,x_{2:p}),\eta_{\star};\bayes))\\
\times \prod_{k=3}^{p+2} \condlik(x_{k-1},v_{k-1},v_k)\pi^{\otimes p}(\rmd v_{3:p+2})\eqsp.
\end{multline*}
By H\ref{assum:lowerbound}, the function $\mathsf{h}$ defined in \eqref{eq:def:h} satisfies:
\begin{equation}
\label{eq:deltah}
\left|\mathsf{h}(v,x,\eta,\eta_{\star})-\mathsf{h}(v,x,\eta',\eta_{\star}')\right|\leqslant \nu^{-1}\left(\|\eta-\eta'\|_{\mathsf{tv}}+\|\eta_{\star}-\eta_{\star}'\|_{\mathsf{tv}}\right)\eqsp.
\end{equation}
The exponential forgetting property of the predictive filter given by Lemma~\ref{lem:predictive:forgetting} yields
\begin{equation}
\label{eq:deltaphi}
\|\Phi_{p-1}(x_{2:p},\Phi(x,\eta;\pi);\pi) - \Phi_{p-1}(x_{2:p},\Phi(x,\eta';\pi);\pi)\|_{\mathsf{tv}} \leqslant (1-\nu)^{p-1}
\end{equation}
and, as $\Phi_{p}((x,x_{2:p}),\eta;\pi) = \Phi_{p-1}(x_{2:p},\Phi(x,\eta;\pi);\pi)$, 
\[
\left|Q_{\pi}^p\mathsf{h}(v,x,\eta,\eta_{\star})-Q_{\pi}^p\mathsf{h}(v,x,\eta',\eta_{\star}')\right| \leqslant \frac{2}{\nu}(1-\nu)^{p-1}\eqsp.
\]
For the second term of \eqref{eq:decomp:Qph}, write for $p\geqslant 2$, 
\[
\Phi_{p}((x,x_{2:p}),\eta;\pi) = \Phi_{p-2}(x_{3:p},\Phi_{2}((x,x_2),\eta;\pi);\pi)
\]
and for any $\phi$ and $\phi_{\star}$ in $\mathcal{P}$,
\begin{align*}
h(v_{p+2},x_{p+1},\Phi_{p}((x,x_{2:p}),\eta;\pi),\Phi_{p}((x,x_{2:p}),\eta_{\star};\bayes)) &=\\
&\hspace{-6cm} h(v_{p+2},x_{p+1},\Phi_{p-2}(x_{3:p},\Phi_{2}((x,x_{2}),\eta;\pi);\pi),\Phi_{p-2}(x_{3:p},\Phi_{2}((x,x_{2}),\eta_{\star};\bayes);\bayes))\\
&\hspace{-2cm}-h(v_{p+2},x_{p+1},\Phi_{p-2}(x_{3:p},\phi;\pi),\Phi_{p-2}(x_{3:p},\phi_{\star};\bayes))\\
&\hspace{-2cm}+h(v_{p+2},x_{p+1},\Phi_{p-2}(x_{3:p},\phi;\pi),\Phi_{p-2}(x_{3:p},\phi_{\star};\bayes))\eqsp.
\end{align*}
Then using that
\begin{align*}
&\sum_{x_{2:p+1}\in\cX^{p}}\int h(v_{p+2},x_{p+1},\Phi_{p-2}(x_{3:p},\phi;\pi),\Phi_{p-2}(x_{3:p},\phi_{\star};\bayes))
  \prod_{k=4}^{p+2} \condlik(x_{k-1},v_{k-1},v_k)  \condlik(x_{2},v,v_3)\pi^{\otimes p}(\rmd v_{3:p+2})
  \\
&=
\sum_{x_{2:p+1}\in\cX^{p}}\int h(v_{p+2},x_{p+1},\Phi_{p-2}(x_{3:p},\phi;\pi),\Phi_{p-2}(x_{3:p},\phi_{\star};\bayes))
  \prod_{k=4}^{p+2} \condlik(x_{k-1},v_{k-1},v_k)  \condlik(x_{2},v',v_3)\pi^{\otimes p}(\rmd v_{3:p+2})
\end{align*}
and \eqref{eq:deltah} and \eqref{eq:deltaphi},
\begin{align*}
\left|Q_{\pi}^p\mathsf{h}(v,x,\eta',\eta_{\star}')-Q_{\pi}^p\mathsf{h}(v',x',\eta',\eta_{\star}')\right|&\\
&\hspace{-5cm}\leqslant 2\nu^{-1}(1-\nu)^{p-2}\!\!\!\!\!\sum_{x_{2:p+1}\in\cX^{p}}\int 
\left(K(x_{2},v,v_3)+K(x_{2},v',v_3)\right)\prod_{k=4}^{p+2} K(x_{k-1},v_{k-1},v_k)\pi^{\otimes p}(\rmd v_{3:p+2})\\
&\hspace{-5cm}\leqslant 4\nu^{-1}(1-\nu)^{p-2} \eqsp.
\end{align*}
In the case $p=1$,
\[
\left|Q_{\pi}^p\mathsf{h}(v,x,\eta',\eta_{\star}')-Q_{\pi}^p\mathsf{h}(v',x',\eta',\eta_{\star}')\right|\leqslant 4\nu^{-1}\eqsp.
\]
\end{proof}

\section{Concentration inequality for Markov chains with uniformly lower bounded transition kernel}
Theorem~\ref{th:conc:gen} and its proof can be found in \cite{TheChampions:2054}. As this book is not published yet, the proof of this result is reproduced here for the sake of completeness.
Let $\nu>0$ and $(\sfX,B(\sfX))$ be a measurable space. In this section, $P$ denotes a Markov kernel defined on $(\sfX,B(\sfX))$ bounded from below by $\nu>0$ : there exists a probability measure $\eta$ on $(\sfX,B(\sfX))$ such that for all $A\in B(\sfX)$ and all $x\in\sfX$, 
\begin{equation}
\label{eq:def:nu:eta}
P(x,A)\geqslant \nu \eta(A)\eqsp.
\end{equation}
For any $x\in \sfX$ (resp. any probability distribution $\pi$ defined on $(\sfX,B(\sfX))$), let $\sfP_x$ (resp. $\sfP_{\pi}$) be the distribution of a Markov chain $(X_n)_{n\geqslant 0}$ with Markov kernel $P$ and such that $X_0=x$ (resp. $X_0\sim \pi$).
The expectation w.r.t. $\sfP_x$ (resp. $\sfP_{\pi}$) is denoted by $\sfE_x$ (resp. $\sfE_{\pi}$).
Let $\gamma\in \rset_+^n$ and $B(\gamma)$ be the set of functions $f:\sfX\to \rset$ such that, for all $x,y\in \sfX^n$, $k\in \{1,\ldots,n\}$,
\[
|f(x)-f(y)|\leqslant \sum_{k=1}^n\gamma_k\un{x_k\ne y_k}\eqsp.
\]
Following for instance \cite[Chapter~4]{cappe:moulines:ryden:2005}, the Dobrushin coefficient of $P$ can be defined as 
\[
\Delta(P)=\sup_{(x,x')\in \sfX^2} \|P(x,\cdot) - P(x',\cdot)\|_{\mathrm{tv}}=\sup_{\pi\ne\pi'}\frac{\|\pi P - \pi'P\|_{\mathrm{tv}}}{\|\pi - \pi'\|_{\mathrm{tv}}}\eqsp,
\]
where for any distribution $\pi$ on $(\sfX,B(\sfX))$ and all $A\in B(\sfX)$, $\pi P (A) = \int \pi(\rmd x) P(x,A)$. For all $n\geqslant 1$, define
\begin{equation}
\label{eq:defDn}
D_n=\sum_{\ell=1}^{n}\pa{\gamma_\ell+2\sum_{m=\ell+1}^{n}\gamma_{m}\Delta\left(P^{m-\ell}\right)}^2\eqsp.
\end{equation}
\begin{theorem}
\label{th:conc:gen}
For any distribution $\pi$ on $(\sfX,B(\sfX))$, any $\gamma\in\rset^n_+$, any $f\in B(\gamma)$ and any $t>0$,
\[
\sfP_{\pi}\pa{f(X_1,\ldots,X_n)-\E_{\pi}\left[f(X_1,\ldots,X_n)\right]>t}\leqslant e^{-2t^2/D_n}\eqsp,
\]
where $D_n$ is defined by \eqref{eq:defDn}.
\end{theorem}

\begin{proof}
Assume without loss of generality that $\sfE_{\pi}[f(X_1,\ldots,X_n)]=0$. For any $\ell\in\{1,\ldots,n\}$, define
\[
A_{\ell}=\gamma_\ell+2\sum_{m=\ell+1}^n\gamma_m\Delta(P^{m-\ell})\eqsp.
\]
For all $\ell\geqslant 0$, define $\cF_\ell=\sigma(X_i,1\leqslant i\leqslant \ell)$ and $G_{\ell}=\sfE_\pi[f(X_1,\ldots,X_n)|\cF_\ell]$ with the conventions $\cF_0=\{\emptyset, \Omega)$ and $G_0=0$. Hence,
\[
f(X_1,\ldots,X_n)=\sum_{\ell=1}^n(G_\ell-G_{\ell-1}),\qquad G_{\ell-1}=\E[G_{\ell}|\cF_{\ell-1}]\eqsp.
\] 
The pivotal step of the proof is to establish that for all $1\leqslant \ell \leqslant n$ there exist $\cF_{\ell-1}$-measurable random variables $a_\ell$ and $b_\ell$ such that
\begin{equation}\label{eq:boundG}
a_\ell\leqslant G_\ell\leqslant b_\ell \quad\mbox{and}\quad b_\ell-a_{\ell}\leqslant A_\ell\eqsp.
\end{equation}
For all $1\leqslant \ell\leqslant n$, $\ell\leqslant s\leqslant n-1$, $x_{1:\ell}\in \sfX^\ell$ and $x_{s+1:n}\in\sfX^{n-s}$ define
\[
K_{\ell,s}(x_1,\ldots,x_{\ell},x_{s+1},\ldots,x_n)=\inf_{(u_{\ell+1},\ldots,u_s)\in\sfX^{s-\ell}}f(x_1,\ldots,x_{\ell},u_{\ell+1},\ldots,u_{s},x_{s+1},\ldots,x_n)
\]
and
\[
W_{\ell,s}(x_1,\ldots,x_{\ell},x_{s+1},\ldots,x_n)=K_{\ell,s}(x_1,\ldots,x_{\ell},x_{s+1},\ldots,x_n)-K_{\ell,s+1}(x_1,\ldots,x_{\ell},x_{s+2},\ldots,x_n)\eqsp.
\]
Hence, $0\leqslant W_{\ell,s}\leqslant \gamma_{s+1}$ and 
\[
 G_\ell=K_{\ell,n}(X_1,\ldots,X_\ell)+\sum_{s=\ell}^{n-1}\sfE_\pi[W_{\ell,s}(X_1,\ldots,X_\ell,X_{s+1},\ldots,X_n)|\cF_\ell]\eqsp.
 \]
 Let $X_1^*,\ldots,X_n^*$ denote an independent copy of $X_1,\ldots,X_n$ so that
 \[
 G_\ell=K_{\ell,n}(X_1,\ldots,X_\ell)+\sum_{s=\ell}^{n-1}\sfE_{\pi}[W_{\ell,s}(X_1,\ldots,X_\ell,X^*_{s+1},\ldots,X^*_n)|\cF_\ell]\eqsp.
 \]
Defining $H_\ell=\sfE_\pi[f(X_0,\ldots,X_\ell,X_{\ell+1}^*,\ldots,X_n^*)|\cF_\ell]$ yields
\[
\left|G_\ell-H_\ell\right|\leqslant \sum_{s=\ell}^{n-1}\left|\sfE_\pi[W_{\ell,s}(X_1,\ldots,X_\ell,X_{s+1},\ldots,X_n)|\cF_\ell]-\sfE_{\pi}[W_{\ell,s}(X_1,\ldots,X_\ell,X^*_{s+1},\ldots,X^*_n)|\cF_\ell]\right|\eqsp.
\]
Write for all $y\in\sfX$,
\[
h(y)=\sfE_{\pi}[W_{\ell,s}(X_1,\ldots,X_{\ell},y)|\cF_\ell]\quad\mbox{and}\quad h'(y)=\sfE_y[h(X_1,\ldots,X_{n-s})]\eqsp. 
\]
Then,
\begin{align*}
 |\sfE_\pi[W_{\ell,s}(X_1,\ldots,X_\ell,X_{s+1},\ldots,X_n)|\cF_\ell]-\sfE_\pi[W_{\ell,s}(X_1,\ldots,X_\ell,X^*_{s+1},\ldots,X^*_n)|\cF_\ell]|&\\
& \hspace{-8cm} =|\sfE_{X_\ell}[h(X_{s+1-\ell},\ldots,X_{n-\ell})]-\sfE_{\pi P^\ell}[h(X_{s+1-\ell},\ldots,X_{n-\ell})]|\eqsp,\\
& \hspace{-8cm} =\int \pi P^\ell(\rmd y)(P^{s+1-\ell}h'(X_{\ell})-P^{s+1-\ell}h'(y))\eqsp,\\
& \hspace{-8cm} \leqslant \norm{h'}_{\infty}\Delta(P^{s+1-\ell})\eqsp.
\end{align*}
As $\norm{h'}_{\infty}\leqslant \norm{W_{\ell,s}}_{\infty}\leqslant \gamma_{s+1}$, we have 
\[
| G_\ell-H_\ell|\leqslant \sum_{s=\ell+1}^{n}\gamma_{s}\Delta(P^{s-\ell})\enspace.
 \]
 Define, for all $1\leqslant \ell\leqslant n$, the following $\cF_{\ell-1}$-measurable random variables
\begin{align*}
 a_\ell'&=\inf_{x_\ell\in \sfX}\E[f(X_1,\ldots,X_{\ell-1},x_{\ell},X_{\ell+1}^*,\ldots,X_n^*|\cF_\ell]\eqsp,\\
 b_\ell'&=\sup_{x_\ell\in \sfX}\E[f(X_1,\ldots,X_{\ell-1},x_{\ell},X_{\ell+1}^*,\ldots,X_n^*|\cF_\ell]\eqsp.
\end{align*}
Since $a_{\ell}'\leqslant H_{\ell}\leqslant b_{\ell}'$ and $b_{\ell}'-a_{\ell}'\leqslant \gamma_{\ell}$, it follows that
\begin{gather*}
 a_\ell=a'_{\ell}-\sum_{s=\ell+1}^{n}\gamma_{s}\Delta(P^{s-\ell}) \quad\mbox{and}\quad b_\ell=b'_{\ell}+\sum_{s=\ell+1}^{n}\gamma_{s}\Delta(P^{s-\ell})
\end{gather*}
satisfy \eqref{eq:boundG}. By McDiarmid's inequality, for all $\lambda>0$,
\[
8\log \sfE_\pi[e^{\lambda(G_{\ell}-G_{\ell-1})}|\cF_{\ell-1}]\leqslant \lambda^2A_\ell^2\eqsp.
\]
Then, by Markov's inequality, for any $s>0$,
\[
\sfP_{\pi}\left(f(X_1,\ldots,X_n)>t\right)\leqslant e^{-\lambda t+\frac{\lambda^2}8\sum_{\ell=1}^nA_{\ell}^2}=\rme^{-\lambda t+\frac{\lambda^2D_n}8}\eqsp.
\]
Applying this inequality with $\lambda =4t/D_n$ concludes the proof.

\end{proof}

\begin{corollary}
\label{cor:ConcULBMC}
For any distribution $\pi$ on $\sfX$, any $\gamma\in\rset^n_+$, any $f\in B(\gamma)$ and any $t>0$,
\[
 \sfP_{\pi}\pa{f(X_1,\ldots,X_n)-\E_{\pi}[f(X_1,\ldots,X_n)]>t} \leqslant \rme^{-\frac{\nu^2}{5}\frac{t^2}{\sum_{\ell=1}^{n}\gamma_\ell^2}}\eqsp.
\]
\end{corollary}

\begin{proof}
For all $x\in\sfX$, define $Q(x,\cdot)=(1-\nu)^{-1}(P(x,\cdot)-\nu\eta(\cdot))$  where $\nu$ and $\eta$ are given by \eqref{eq:def:nu:eta}. For any $x\in \sfX$, $Q(x,\cdot)$ is a probability distribution and
\[
P(x,\cdot)=(1-\nu)Q(x,\cdot)+\nu\eta(\cdot)\eqsp.
\]
Therefore, for all $(x,x')\in\sfX^2$,
\[
\|P(x,\cdot) - P(x',\cdot)\|_{\mathrm{tv}}=(1-\nu)\|Q(x,\cdot) - Q(x',\cdot)\|_{\mathrm{tv}}\leqslant 1-\nu\eqsp,
\]
which yields, for all $q\geqslant 1$, $\Delta(P^q)\leqslant (\Delta(P))^q\leqslant (1-\nu)^q$. By Theorem~\ref{th:conc:gen}, for any distribution $\pi$ on $\sfX$, any $\gamma\in\rset^n_+$, any $f\in B(\gamma)$ and any $t>0$,
\begin{equation*}
\sfP_{\pi}\pa{f(X_1,\ldots,X_n)-\E_{\pi}[f(X_1,\ldots,X_n)]>t}\leqslant \rme^{-\frac{2t^2}{\sum_{\ell=1}^{n}\pa{\gamma_\ell+2\sum_{m=\ell+1}^{n}\gamma_{m}(1-\nu)^{m-\ell}}^2}}\eqsp.
\end{equation*}
Note that by Cauchy-Schwarz inequality,
\begin{align*}
\sum_{\ell=1}^{n}\left(\gamma_\ell+2\sum_{m=\ell+1}^{n}\gamma_{m}(1-\nu)^{m-\ell}\right)^2 &\leqslant 2 \sum_{\ell=1}^{n}\gamma^2_\ell + 8 \sum_{\ell=1}^{n}\left(\sum_{m=\ell+1}^{n}\gamma_{m}(1-\nu)^{m-\ell}\right)^2\eqsp,\\
&\leqslant 2 \sum_{\ell=1}^{n}\gamma^2_\ell + 8 \frac{1-\nu}{\nu}\sum_{\ell=1}^{n}\left(\sum_{m=\ell+1}^{n}\gamma^2_{m}(1-\nu)^{m-\ell}\right)\eqsp,\\
&\leqslant 2\left(1 + 4\frac{(1-\nu)^2}{\nu^2} \right) \sum_{\ell=1}^{n}\gamma^2_\ell\eqsp,
\end{align*}
which concludes the proof.
\end{proof}

\bibliographystyle{plain}
\bibliography{bayesianbt}

\end{document}